\documentclass[12pt, a4paper]{article} 

\usepackage{amssymb, amsmath, amsfonts, amsthm}
\usepackage[T1, T2A]{fontenc}
\usepackage[utf8]{inputenc}
\usepackage[english]{babel}
\usepackage{fancyhdr}
\usepackage{latexsym}
\usepackage{mathtools, enumerate, rotating, framed, lipsum}
\usepackage{graphicx}
\usepackage[pass]{geometry}
\usepackage{wrapfig}
\usepackage{rotating}
\usepackage{textcomp}
\usepackage{arcs}
\usepackage{gensymb}
\usepackage{mathrsfs}
\usepackage{multicol}
\usepackage{titlesec}
\usepackage{chngcntr}
\usepackage{xcolor}
\usepackage{float}
\usepackage{tikz}
\usepackage{courier}
\usepackage{thmtools}
\usepackage[hidelinks]{hyperref}
\usepackage{yfonts}
\usepackage{enumitem}
\usepackage{parskip}
\usepackage{etoolbox}
\usepackage{setspace}
\usepackage{relsize}
\usepackage{csquotes}

\usepackage[shortcuts]{extdash}

\usetikzlibrary{decorations.markings}

\newcommand{\NN}{\mathbb{N}} \newcommand{\ZZ}{\mathbb{Z}}
 
\newcommand{\8}{\infty}

\newcommand{\abs}[1]{\left| #1 \right|}

\newcommand{\set}[1]{\left\{ #1 \right\}}

\DeclarePairedDelimiter\floor{\lfloor}{\rfloor}

\DeclareMathOperator{\Cay}{Cay}
\DeclareMathOperator{\diam}{diam}

\DeclareMathAlphabet\mathbfcal{OMS}{cmsy}{b}{n}

\makeatletter
\def\th@plain{%
  \thm@notefont{}
  \itshape 
}
\def\th@definition{%
  \thm@notefont{}
  \normalfont 
}
\makeatother

\newtheoremstyle{break}
  {\topsep}{\topsep}%
  {\itshape}{}%
  {\bfseries}{}%
  {\newline}{\thmname{#1}\thmnumber{ #2}\thmnote{ (#3)}}%
\theoremstyle{break}
\newtheorem{thm}{Theorem}[section]
\newtheorem{lm}[thm]{Lemma}

\newtheorem{fact}[thm]{Fact}
\newtheorem{cor}[thm]{Corollary}
\newtheorem{oprob}[thm]{Question}
\newtheorem{rk}[thm]{Remark}

\newtheorem*{thma*}{Theorem A}
\newtheorem*{thmb*}{Theorem B}

\newtheoremstyle{defbreak}
  {\topsep}{\topsep}%
  {\upshape}{}%
  {\bfseries}{}%
  {\newline}{\thmname{#1}\thmnumber{ #2}\thmnote{ (#3)}}%
\theoremstyle{defbreak}
\newtheorem{df}[thm]{Definition}

\renewcommand{\thethm}{%
  \ifnum\value{subsection}=0 
    \thesection
  \else
    \thesubsection
  \fi%
  .\arabic{thm}} 
  
\newcommand{\ol}[1]{\overline{#1}}
\newcommand{\wtl}[1]{\widetilde{#1}}

\renewcommand{\bigstar}{\mathop{\mathlarger{\mathlarger{\mathlarger{*}}}}}

\DeclareMathOperator*{\damalgam}{\widetilde{\bigsqcup}}

\newcommand{\leqnomode}{\tagsleft@true\let\veqno\@@leqno}

\newcounter{todocounter}
\addtocounter{todocounter}{1}

\newcommand{\com}[1]{}

\title{$E\mathcal{Z}$-boundaries, splittings over finite subgroups, \\and dense amalgams}
\author{Mateusz Kandybo and Jacek Świątkowski}
\begin{document}


\maketitle

\begin{abstract}
The {\it dense amalgam} is an operation (introduced in \cite{dense-amalgam}) which to any finite collection of metrizable compacta associates canonically some new highly disconnected compact metrisable space in which embedded copies of the initial spaces are appropriately uniformly and disjointly distributed. We show that in the very general framework of $E\mathcal{Z}$-boundaries (unifying many frameworks such as Gromov boundaries,
CAT(0)-boundaries, systolic boundaries, etc.),
any boundary $Z$ of an infinitely ended group $\Gamma$ equipped with an appropriate splitting along finite subgroups has a form of the dense amalgam of the limit sets in $Z$ of the factor subgroups of this splitting.
\end{abstract}

\section{Introduction}

It is a classical and widely studied problem, to relate algebraic features
of discrete groups with some corresponding topological properties of their
boundaries at infinity. In this paper we study such a relationship
for the property of admitting a nontrivial splitting over finite subgroups.

As it was shown in \cite{dense-amalgam}, in many frameworks
(such as Gromov boundaries of hyperbolic groups, CAT(0) boundaries,
systolic boundaries, etc.) groups that split over finite subgroups admit
boundaries that have some particular form of so called \textit{dense amalgam}
(of the boundaries of the factors).
Dense amalgam is a topological operation (introduced in \cite{dense-amalgam}
and recalled in Subsection \ref{Subsection:dense_amalgams} below) which to any finite collection
of compact metrisable spaces associates some new (unique up to homeomorphism)
compact metrisable and highly disconnected space in which infinitely many copies of the initial spaces
appear as appropriately uniformly distributed subspaces (see Definition \ref{Definition:dense_amalgam}).
The goal of this paper is the twofold extension of the aforementioned observation. First, we show that in the contexts as above 
\textit{any} boundary (of the corresponding kind) of a group admitting
a splitting over finite subgroups has a form of the dense amalgam of some
spaces associated naturally to the factors. This is indeed an extension,
since a given group can have many pairwise non-homeomorphic boundaries
of a given kind, and we show that \textit{all} of them have a form of a 
dense amalgam. Second, we justify the just mentioned ``dense amalgam''
phenomenon in a unified framework of 
so called $E\mathcal{Z}$-boundaries of groups (as recalled in Subsection \ref{Subsection:EZ-boundaries}).
The latter framework encompasses and extends all the previously
mentioned frameworks.

The precise statement of our main result requires few more technical
and terminological explanations. First of all, we encode splittings of groups 
$\Gamma$ over their finite subgroups by means of expressions of
these groups as the fundamental groups $\pi_1\mathcal{G}$ of appropriate
graphs of groups $\mathcal{G}$. The vertex groups $G_v$ of $\mathcal{G}$
are then viewed naturally as subgroups of $\Gamma$, and they are the factors
of the considered splittings. See Subsection \ref{Subsection:graphs_of_groups} for more detailed explanations, and in particular see Definition \ref{Definition:non-elementary_gog} for the explanation of the concept
of a non-elementary graph of groups.
Next, we refer to the notion of the limit set of a subgroup for a group
equipped with an $E\mathcal{Z}$-boundary. The details of this concept
are recalled in Subsection \ref{Subsection:EZ-boundaries}. 
We denote by $V_Y$ the vertex set of a graph $Y$.
Our main result then reads as follows.

\bigskip
\begin{thma*}\label{thma}
Let $(\mathcal{G},Y)$ be a non-elementary graph of groups (over a finite graph $Y$) with finite
edge groups, let $\Gamma=\pi_1\mathcal{G}$ be its fundamental group,
and let $\set{G_v:v\in V_Y}$ be the family of the vertex groups of $\mathcal{G}$, which we assume to be finitely generated
(and which we view as subgroups of $\Gamma$). Let $(\ol{E},Z)$ be an
$E\mathcal{Z}$-structure for $\Gamma$. Then the $E\mathcal{Z}$-boundary
$Z$ of $\Gamma$ (corresponding to this structure) is homeomorphic to
the dense amalgam of the family of the limit sets $\set{\Lambda G_v:v\in V_Y}$
in $Z$, i.e. $$Z\cong\damalgam_{v\in V_Y}\Lambda G_v.$$
\end{thma*}

By referring to Stallings' theorem about ends of groups
and to various further basic properties of splittings of groups over finite subgroups,
Theorem A combined with few other observations from
this paper yields the following corollary (see the last part of Section \ref{Section:main_result} 
for a detailed justification).

\medskip
\begin{thmb*}
Let $\Gamma$ be a finitely presented group, and suppose that
$\Gamma$ admits an EZ-boundary $Z$. Then:
\begin{enumerate}
\item
$\Gamma$ is 1-ended iff $Z$ is connected;
\item
$\Gamma$ is 2-ended iff $Z$ is a doubleton;
\item
$\Gamma$ is infinitely ended iff $Z$ is the dense amalgam of a family of connected spaces. 

\end{enumerate}

\end{thmb*}

The paper is organized as follows.
We present an outline of the proof of Theorem A in Subsection \ref{Subsection:outline_of_proof}
(after completing all preliminary informations in Subsections \ref{Subsection:graphs_of_groups}-\ref{Subsection:dense_amalgams}).
The actual proof of Theorem A occupies Section \ref{Section:main_result}. Sections \ref{Section:separation} and \ref{Section:classification_of_points} 
contain some preparations for this proof. In particular, in Section \ref{Section:separation} we 
introduce tools related to various phenomena of separation in spaces
and groups. Section \ref{Section:classification_of_points} contains some auxiliary analysis and classification of the 
$E\mathcal{Z}$-boundary points
for groups $\Gamma=\pi_1\mathcal{G}$ as in the statement of Theorem A. Lastly, in Section \ref{Section:final_conclusions} we state some final remarks and ask $3$ questions regarding possible expansions of our results from Theorems A and B.

\textbf{Acknowledgements}\\
This paper was created to expand the work from the first author's Bachelor Thesis (written under the supervision of the second author). We would like to thank Jan Dymara (reviewer of the aforementioned thesis) for helpful comments that effectively pushed us to start this project. Moreover we would like to thank Damian Osajda for a discussion about $E\mathcal{Z}$-structures which allowed us to expand and generalize our initial result. The work presented here was partially supported by the Carlsberg Foundation, grant CF23-1226.

\section{Preliminaries} \label{Section:Preliminaries}


\subsection{Graphs of groups}\label{Subsection:graphs_of_groups}

In this subsection we recall some definitions and basic facts from Bass-Serre theory. The reader is encouraged to consult the book \cite{serre-trees} for more complete exposition, or Section $4$ of \cite{dense-amalgam}, where this material is treated in the context related to the  aims of this paper.

For any graph $Y$ (we allow graphs to have loops and multiple edges) we denote the set of vertices of $Y$ as $V_Y$, the set of all (unoriented) edges of $Y$ as $\abs{O}_Y$. To each edge $\abs{y}\in \abs{O}_Y$ we associate two oppositely oriented edges underlied by $\abs{y}$. Moreover we denote the set of all oriented edges of $Y$ as $O_Y$. For any oriented edge $y\in O_Y$ we denote its initial vertex as $\alpha(y)$, its final vertex as $\omega(y)$, the oppositely oriented edge as $\ol{y}$ and the corresponding non-oriented edge as $\abs{y}$.\medskip

\begin{df}[Paths and branches]
    An \textit{immersed path} in a graph $Y$ (or shortly a \textit{path}) is a (finite or infinite) sequence $c=(c_1,c_2,\dots)$ of oriented edges such that each pair of consecutive edges satisfies $\omega(c_i) = \alpha(c_{i+1})$ and $c_{i+1}\ne\ol{c_i}$. The \textit{length} $l(c)$ of a path $c$ is the number of edges in this path. The \textit{initial} vertex of a path, denoted $\alpha(c)$ is the initial vertex of its first edge, i.e. $\alpha(c)\coloneqq\alpha(c_1)$. In case of finite paths, the \textit{final} vertex of a path $\omega(c)$ is the final vertex of its last edge, i.e. $\omega(c)\coloneqq\omega(c_{l(c)})$. When $Y$ is a tree, any infinite path in $Y$ is called a \textit{branch}.
\end{df}


\begin{df}[Graph of groups]
    A \textit{graph of groups} $(\mathcal{G},Y)$ is a graph $Y$ together with a family of groups $\mathcal{G}$ indexed by all vertices of $Y$ and all oriented edges of $Y$ together with a set of group monomorphisms described below. To differentiate between the family $\mathcal{G}$ and groups being its elements we use a standard notation where $\mathcal{G}$ denotes the whole family while $G_v, G_y$ are elements of that family indexed respectively by $v, y$ (note the lack of the calligraphic font for groups). The group monomorphisms in the family $\mathcal{G}$ are $i_y:G_y \to G_{\omega(y)}$ and they are defined for every $y\in O_Y$; it is also assumed that $G_y = G_{\ol{y}}$. 

    In this paper we will always assume that the underlying graph $Y$ of a graph of groups $\mathcal G$ is finite, and that the vertex groups $G_v$ are finitely generated, even though the next definitions in this subsection do not require these assumptions.
\end{df}

The following definition can be found in Chapter 5.1 of \cite{serre-trees}.\medskip

\begin{df}[Fundamental group of a graph of groups]\label{def:fundamental_group_of_gog}
    Fix a spanning tree $T$ of $Y$. Define the \textit{fundamental group} $\pi_1(\mathcal{G},Y,T)$ as the quotient of the free product $\displaystyle{F_{O_Y\setminus O_T}*\bigstar_{v\in V_Y}G_v}$ by the normal subgroup generated by all relations of one of the following form:
    $$i_y(g)^{-1}i_{\ol{y}}(g) = 1\text{ where }y\in O_T\text{ and }g\in G_y$$
    as well as, for $s_y$ denoting the generator of $F_{O_Y\setminus O_T}$ corresponding to an edge $y$,
    $$s_ys_{\ol{y}} = 1\text{ where }y\in O_Y\setminus O_T,$$
    $$i_y(g)^{-1}s_{y}^{-1}i_{\ol{y}}(g)s_{y} = 1\text{ where }y\in O_Y\setminus O_T,g\in G_y.$$
    It was proven in Chapter 5.2 of \cite{serre-trees} that each of the vertex groups $G_v$  
    naturally embeds into $\pi_1(\mathcal{G}, Y, T)$ as a subgroup. 
    Using these facts, we will abuse notation treating all $G_v$ as subgroups of $\pi_1(\mathcal{G},Y,T)$. Furthermore, it was shown in Chapter 5.1 of \cite{serre-trees} that any choice of a spanning tree $T$ would yield an isomorphic $\pi_1(\mathcal{G},Y,T)$, therefore we refer to it as the fundamental group of $\mathcal{G}$.
\end{df}

The fundamental groups of graphs of groups provide a convenient way of describing groups obtained as iterated free products with amalgamation and HNN-extensions. The main advantage of using the graph of groups expression comes from existence of the so-called Bass-Serre tree -- a structure on which the corresponding fundamental group acts in a manageable way. In our case, Bass-Serre trees will be useful for constructing appropriate separating sets in spaces and graphs appearing in later parts of the paper. We present the definition of a Bass-Serre tree below.  More details and  context concerning this definition can be found in Chapter 5.3 of \cite{serre-trees}.\medskip

\begin{df}[Bass-Serre tree]
     Let $(\mathcal{G},Y)$ be a graph of groups together with a fixed spanning tree $T$ and a fixed orientation of edges in $Y$ (i.e. a set $A$ such that for each pair of edges $y,\ol{y}$ exactly one of those edges belongs to $A$). We will denote the unique edge not in $A$ from any pair $\set{y, \ol{y}}$ as $\hat{y}$. Then, for each edge $y$ we have $G_y = G_{\hat{y}} \hookrightarrow G_{\omega(\hat{y})}<\Gamma \coloneqq \pi_1(\mathcal{G},Y,T)$. With that we can treat all the edge groups of $(\mathcal{G},Y)$ as subgroups of $\Gamma$. A \textit{Bass-Serre tree} of a graph of groups $(\mathcal{G},Y)$ is a graph $\wtl{X} = \wtl{X}(\mathcal{G},Y,T)$ whose sets of vertices and edges are respectively  $V_{\wtl{X}} = \bigsqcup_{v\in V_Y}\Gamma/G_v$ and $O_{\wtl{X}} = \bigsqcup_{y\in O_Y}\Gamma/G_{y}$. 
     It remains to define $\omega(\gamma G_y)$ and $\ol{\gamma G_y}$ for each edge $\gamma G_y\in O_{\wtl{X}}$, which is done as follows:
     \begin{itemize}
         \item $\ol{\gamma G_{y}} = \gamma G_{\ol{y}}$,
         \item $\omega(\gamma G_{y}) = \begin{cases}
             \gamma s_{y}G_{\omega(y)}\text{, when }y\in A\setminus O_T\\
             \gamma G_{\omega(y)}\text{ otherwise}
         \end{cases}$.
     \end{itemize}
     It is shown in a Chapter 5.3 of \cite{serre-trees} that a graph constructed in such a way is a tree. Moreover $\wtl{X}$ is equipped with a natural action by $\Gamma$:
     $$\gamma'\cdot (\gamma G_v) = \gamma'\gamma G_v \text{ and }\gamma'\cdot \gamma G_{y} = \gamma'\gamma G_{y}.$$
\end{df}

\begin{rk}\label{remark:containment_of_edge}
    Note that, with the definition as above, the coset $\gamma G_y$ representing any edge in a Bass-Serre tree $\wtl{X}$ is contained in the union of cosets $\alpha(\gamma G_y) \cup \omega(\gamma G_y)$ representing the endpoints of this edge. Indeed
    $$\gamma G_y = \gamma G_{\hat{y}} \subseteq \gamma G_{\omega(\hat{y})} = \omega(\gamma G_{\hat{y}}) \subseteq \alpha(\gamma G_{\hat{y}})\cup \omega(\gamma G_{\hat{y}}) = \alpha(\gamma G_{y})\cup \omega(\gamma G_{y}).$$
\end{rk}

The next two definitions aim to provide an appropriate context for the technical concept of a {\it non-elementary graph of groups}, that appears in the statement of our first main result (Theorem A of the introduction). The primary source for this concept is Subsection 4.1 of the paper \cite{dense-amalgam}.\medskip

\begin{df}[Elementary collapse]
    Let $(\mathcal{G},Y)$ be a graph of groups and let $y\in O_Y$ be an edge that is not a loop such that the map $i_y$ is an isomorphism. An \textit{elementary collapse} along the edge $y$ is an operation returning a new graph of groups $(\mathcal{G}',Y')$, where $Y'$ is a graph obtained from $Y$ by contracting the edge $y$ to a point (denoted $v_y$) and by setting $G_{v_y} = G_{\alpha(y)}$ and not changing all of the other groups and maps.
\end{df}

Note that elementary collapses do not change the fundamental group of a graph of groups.\medskip

\begin{df}[Non-elementary graph of groups]\label{Definition:non-elementary_gog}
    We say that a graph of groups $(\mathcal{G},Y)$ is \textit{simply elementary} if it satisfies one of the three conditions:
    \begin{itemize}
        \item $Y$ consists of a single vertex without any edges;
        \item $Y$ consists of two vertices connected by a unique non-oriented edge $\abs{y}$ and the images of both maps $i_y, i_{\ol{y}}$ are the subgroups of index 2
        in the corresponding vertex groups;
        \item $Y$ consists of a single vertex and a single non-oriented loop $\abs{y}$ and both of the maps $i_y, i_{\ol{y}}$ are isomorphisms.
    \end{itemize}
    A graph of groups is \textit{non-elementary} if it can't be reduced to a simply elementary graph via elementary collapses.
\end{df}

Some consequences of being a non-elementary graph of groups are provided in Lemma 4.1.6 in \cite{dense-amalgam}. 
Since we will need them in our arguments
 for proving Theorem A, we restate this lemma below
 for reader's convenience.\medskip

\begin{lm}\label{lemma:behaviour_of_nonelementary_gog}
    Let $(\mathcal{G},Y)$ be a non-elementary graph of groups with all of the edge groups being finite. Let $\Gamma\coloneqq\pi_1(\mathcal G, Y, T)$, and let $\wtl{X} \coloneqq\wtl{X}(\mathcal{G}, Y, T)$ be the Bass-Serre tree.
    \begin{itemize}
        \item[(i)] For every vertex $v\in V_Y$ with infinite $G_v$, there exists an edge $y\in O_Y$ adjacent to $v$ and such that any edge of form $\abs{\gamma G_y}$ splits $\wtl{X}$ into two subtrees, each of which contains at least one vertex of form $\gamma' G_{v'}$ for any vertex $v'\in V_Y$ and any $\gamma'\in\Gamma$;
        \item[(ii)] If all of the vertex groups of $\mathcal{G}$ are finite, then 
        $\wtl{X}$ is an infinite locally finite tree and
        there exists $v\in V_Y$ such that for any $\gamma\in\Gamma$ the vertex
        $\gamma G_v$ splits $\wtl{X}$ into at least $3$ infinite components.
    \end{itemize}
\end{lm}


\subsection{\texorpdfstring{$\mathbf{E}\mathbfcal{Z}$}{EZ}-boundaries
and limit sets}\label{Subsection:EZ-boundaries}

We present the definition of an 
$E\mathcal{Z}$-structure and 
an $E\mathcal{Z}$-boundary, in a version appropriate for the purposes of
this paper (and slightly more general than the versions appearing in the literature,
see e.g. \cite{Farell-Lafont, guilbault}).\medskip

\begin{df}\label{def:EZ-structures} Let $\Gamma$ be a discrete group.
The pair of spaces $(\overline E,Z)$ such that $Z\subseteq \overline E$
is called an $E\mathcal{Z}$-\textit{structure} 
for $\Gamma$ if it satisfies the following conditions:

\item{(1)} $\overline E$ is a compact metrizable space which is connected and locally connected;

\item{(2)} $Z$ is a Z-set of $\overline E$, i.e. it is closed in $\overline E$ 
and there is a deformation
$h_t:\overline E\to\overline E$ with $h_0=\hbox{id}$ and 
$h_t(\overline E)\cap Z=\emptyset$ for $t>0$;

\item{(3)} $E\coloneqq\overline E\setminus Z$ is equipped with a proper and cocompact action by $\Gamma$;

\item{(4)} for any compact $K\subseteq E$ the family of $\Gamma$-translates
of $K$ (under the above action) is {\it null} in $\overline E$ (i.e. for any metric on $\ol E$ compatible with its topology the diameters of the $\Gamma$-translates
of $K$ converge to 0);

\item{(5)} the above action of $\Gamma$ on $E$ extends continuously
to an action on $\overline E$.

The set $Z$ above is called an \textit{$E\mathcal{Z}$-boundary} of $\Gamma$.
\end{df}

\medskip
\begin{rk}\label{rk:Z-set}
Note that, under the assumptions as above, the group $\Gamma$ is automatically finitely generated, and it is finitely presented e.g. when $E$ is additionally a locally finite simply connected CW-complex. The space $E$ above is automatically locally compact, connected, and locally connected.

Note also that if $Z\subseteq\overline E$ is a Z-set, then for any compact subset
$K\subseteq E=\overline E\setminus Z$:

\item{(1)} there is a deformation $h_t:\overline E\to\overline E$ with
$h_0=\hbox{id}$, $h_t|_K=\hbox{id}_K$ for $t\ge0$, and with
$h_t(\overline E)\cap Z=\emptyset$ for $t>0$;

\item{(2)} there is a deformation 
$h_t:\overline E\setminus K\to\overline E\setminus K$ with $h_0=\hbox{id}$
and with $h_t(\overline E\setminus K)\cap Z=\emptyset$ for $t>0$;

\item{(3)} the inclusion $E\setminus K\to\overline E\setminus K$
is a homotopy equivalence.
\end{rk}


The concept of a {\it limit set}, to which we now turn, is well-known in group theory and used e.g. to study Kleinian groups. It is hard to pinpoint the original source of the definition of a limit sets. The definition we present below is based on Definition 2.9 of the paper \cite{martin}, however we state it in slightly more general fashion.\medskip

\begin{df}[Limit set]
    Let $(\ol{E},Z)$ be an $E\mathcal{Z}$-structure for $\Gamma$. The \textit{limit set} of a subset $N\subseteq E$ is the intersection $\ol{N}\cap Z$, where $\ol{N}$ is the closure of $N$ in $\ol{E}$. Such limit set will be denoted as $\Lambda N$. The \textit{limit set of a subset }$M\subseteq\Gamma$ is defined as the limit set $\Lambda Me_0$ of the set of translates $Me_0\coloneqq\{ \gamma e_0:\gamma\in M \}$ 
    of a fixed base point $e_0\in E$. Note that even though this definition involves a base point $e_0$, the limit set itself does not depend on the choice of $e_0$. Indeed: for any two choices of points $e_0, e_0'$ the diameters of the translates of the doubleton $\set{e_0,e_0'}$ converge to $0$, thus for any sequence $m_ne_0$ convergent to a point in $Z$ the corresponding sequence $m_ne_0'$ converges to the same point. Furthermore one can observe that, for $M$ being a subgroup, this definition agrees with the definition from \cite{martin}. We will abuse the notation and denote the limit set of $M\subseteq \Gamma$ as $\Lambda M$.
\end{df}

The following useful observation seems to belong to folklore,
and we include its proof for completeness.
We refer the reader e.g. to Chapter 9.1 in \cite{drutu-kapovich}
for an introduction concerning the concept of ends (and 1-endedness) of a group.\medskip

\begin{lm}\label{limit_set_connected}
Let $\Gamma$ be a group, and let $(\ol E,Z)$ be an  $E\mathcal{Z}$-structure for $\Gamma$. Suppose that $H<\Gamma$ is a finitely generated
subgroup of $\Gamma$ which is 1-ended.
Then the limit set $\Lambda H\subseteq Z$ is connected.
In particular, if $\Gamma$ is 1-ended then $Z=\Lambda\Gamma$
is connected. 
\end{lm}

\begin{proof}
We start with some auxiliary terminology and observation.
Given a metric space $(X,d_X)$ and a real number $\epsilon>0$,
an {\it $\epsilon$-path} in $X$ is any finite sequence 
$x_0,x_1,\dots,x_n$ such that $d_X(x_{i-1},x_i)<\epsilon$
for $i=1,\dots,n$.
We say that $X$ is {\it $\epsilon$-connected} if any two points of $X$
can be joined by an $\epsilon$-path.
It is not hard to observe that a compact metric space is connected
iff it is $\epsilon$-connected for any $\epsilon>0$.
Thus, to show that $\Lambda H$ is connected, it is enough to show that
it is $\epsilon$-connected for any $\epsilon>0$.

Fix any two points $z,z'\in\Lambda H$, and any $\epsilon>0$.
We will construct an $\epsilon$-path joining $z$ with $z'$.
To do this, fix a finite symmetric generating set $S=S^{-1}$ for $H$,
and a base point $e_0\in E$. Fix also a metric $d$ in $\ol E$,
and a compact subset $K\subseteq E$ large enough so that for any pair
$(\gamma,\gamma s):\gamma\in\Gamma, s\in S$ 
such that $\gamma e_0\notin K$ and $\gamma se_0\notin K$ we have
$d(\gamma e_0, \gamma s e_0)<\epsilon/3$
(such a $K$ exists due to condition (4) in Definition 
\ref{def:EZ-structures}).
Put $A\coloneqq\{ \gamma\in H:\gamma e_0\in K \}$ and note that $A$
is a finite subset of $H$. Enlarge $A$ to a new finite subset $A_0\subseteq H$
so that the distance of any point $\gamma e_0:\gamma\in H\setminus A_0$
from the limit set $\Lambda H$ is smaller than $\epsilon/3$.
By 1-endedness of $H$, the complement of $A_0$ in the Cayley graph $\hbox{Cay}(H,S)$
has precisely one unbounded connected component, and by local finiteness of
$\hbox{Cay}(H,S)$, the number of bounded components in this complement is finite.
We can thus fix an even bigger finite
subset $A_0'\subseteq H$ such that any two elements of $H\setminus A_0'$
can be connected
in the Cayley graph $\hbox{Cay}(H,S)$ by a polygonal path disjoint with $A_0$.

Now, choose any $h,h'\in H\setminus A_0'$ such that $d(he_0,z)<\epsilon/3$
and $d(h'e_0,z')<\epsilon/3$. 
Connect the vertices $h,h'$ of $\hbox{Cay}(H,S)$ with a polygonal path
$h=h_0,h_1,\dots,h_n=h'$ disjoint from $A_0$. Then for any $i=1,\dots,n$
we have $h_i=h_{i-1}s_i$ for some $s_i\in S$. 
Since this path omits $A_0$, and hence also omits $A$, the points
$h_ie_0$ are not contained in $K$, and thus we have
$d(h_{i-1}e_0,h_ie_0)=d(h_{i-1}e_0,h_{i-1}s_ie_0)<\epsilon/3$
for all $i=1,\dots,n$. Moreover, by the choice of $A_0$,
for any $i=1,\dots,n-1$ there is a point $z_i\in\Lambda H$ such that
$d(z_i,h_ie_0)<\epsilon/3$.
By referring to the triangle inequality, we get that 
the sequence $z,z_1,\dots,z_{n-1},z'$ is an $\epsilon$-path in 
$\Lambda H$, as required. This finishes the proof.
\end{proof}


\subsection{Dense amalgams}\label{Subsection:dense_amalgams}

Below we present a definition of a dense amalgam which was firstly described in the second author's paper \cite{dense-amalgam}. The definition of the dense amalgam can be found in the introduction of \cite{dense-amalgam} and the proof of its existence and uniqueness can be found in Sections $1$ and $2$ of the same paper.\medskip

\begin{df} \label{Definition:dense_amalgam}
    Let $X_1,\ldots, X_n$ be non-empty compact metrisable spaces. \textit{The dense amalgam} of $X_1, \ldots, X_n$, denoted $\damalgam(X_1,\ldots,X_n)$, is the unique up to homeomorphism compact metrisable space which satisfies the following.
    There exists a countable family $\mathcal{W}$ of subspaces of $\damalgam(X_1,\ldots,X_n)$, which splits as $\mathcal{W} = \mathcal{W}_1\sqcup \ldots \sqcup \mathcal{W}_n$, such that:
    \begin{itemize}
        \item[(a1)] $\mathcal{W}$ consists of pairwise disjoint subspaces 
        and each of the subfamilies $\mathcal{W}_i$ consists of homeomorphic copies of $X_i$, respectively;
        \item[(a2)] the family $\mathcal{W}$ is null (in the same sense as in Definition \ref{def:EZ-structures})
        \item[(a3)] each $W\in \mathcal{W}$ is a boundary subset of $\damalgam(X_1,\ldots,X_n)$ (i.e. the complement of $W$ is dense in $\damalgam(X_1,\ldots,X_n)$);
        \item[(a4)] each union $\bigcup \mathcal{W}_i$ is dense in $\damalgam(X_1,\ldots,X_n)$;
        \item[(a5)] for any two points $x_1, x_2\in \damalgam(X_1,\ldots,X_n)$ that do not belong to the same subset of $\mathcal{W}$ there is a clopen $H\subseteq\damalgam(X_1,\ldots,X_n)$ which is \textit{$\mathcal{W}$-saturated} (i.e. each $W\in \mathcal{W}$ is either a subset of $H$ or is disjoint with $H$) and such that $x_1\in H, x_2\not\in H$.
    \end{itemize}
    For simplicity of notation, we sometimes write $\damalgam_{i\in I}X_i$ instead of $\damalgam(X_1, \ldots, X_n)$, where $I$ is meant here to denote the set $\{ 1,\dots,n \}$.
\end{df}

The aforementioned paper shows a number of algebraic properties of dense amalgam. One of them allows us to generalize the notion of dense amalgam to collections of spaces that are not necessarily all non-empty. We formulate this as a remark below. For additional context see Section $3$ of \cite{dense-amalgam}.\medskip

\begin{rk} \label{Remark:definition_of_dense_amalgam}
    For any non-empty compact metric spaces $X_1, \ldots X_n$ we have
    $$\damalgam(X_1,\ldots, X_n) = \damalgam(X_1\sqcup\ldots\sqcup X_n).$$
    In particular if some (but not all) of the spaces $X_1, \ldots X_n$ are empty we can define
    $$\damalgam(X_1,\ldots,X_n) \coloneqq \damalgam(X_1\sqcup\ldots\sqcup X_n)$$
    or equivalently 
    $$\damalgam(X_1,\ldots,X_n) \coloneqq \damalgam(X_{i_1},\ldots,X_{i_k}),$$
    where $X_{i_1},\ldots,X_{i_k}$ are all of the non-empty sets among $X_1,\ldots,X_n$. Additionally if all of the sets $X_1,\ldots X_n$ are empty we define the dense amalgam of such spaces to be the Cantor space $\mathcal{C}$:
    $$\damalgam(\emptyset,\ldots,\emptyset) = \mathcal{C}.$$
\end{rk}

\subsection{An outline of the proof of Theorem A}\label{Subsection:outline_of_proof}

The special case where all vertex groups are finite is discussed separately (see Lemma \ref{lemma:virtually_free_case} and Corollary  \ref{corollary:all_vertex_groups_finite} below).

We use the notation as in the statement of Theorem A (in the introduction). Put $V_Y^+\coloneqq\{ v\in V_Y:\Lambda G_v\ne\emptyset \}$. By Remark \ref{Remark:definition_of_dense_amalgam},
we need to show that 
\begin{align}
Z\cong\damalgam_{v\in V_Y^+}\Lambda G_v. \label{Equation:Z_characterization}\tag{$*$}
\end{align}
To do this, we will describe an appropriate family $\mathcal{W}$
of subspaces of $Z$ satisfying all relevant requirements
from Definition \ref{Definition:dense_amalgam}. The family will be split as
$\mathcal{W}=\bigsqcup\{ \mathcal{W}_v:v\in V_Y^+ \}$,
with subfamilies $\mathcal{W}_v$ corresponding
to the factors $G_v:v\in V_Y^+$, or alternatively
to the limit sets $\Lambda G_v$ in the expression \eqref{Equation:Z_characterization} above.

For any $v\in V_Y^+$, denoting by $\Gamma/G_v$ the set of left cosets
of $G_v$ in $\Gamma$, we put
$$
\mathcal{W}_v\coloneqq\{ \Lambda C : C\in\Gamma/G_v  \}.
$$
We then put $\mathcal{W}\coloneqq\bigcup\{ \mathcal{W}_v:v\in V_Y^+ \}$.
Obviously, since any limit set is a closed subset of $Z$,
each of the families $\mathcal{W}_v$ consists of
closed subspaces of $Z$. Furthermore, if $C=gG_v$,
we obviously get that $\Lambda C=g\cdot\Lambda G_v$,
and hence all the subspaces in $\mathcal{W}_v$
are homeomorphic to $\Lambda G_v$.

Verification of the remaining requirements of Definition \ref{Definition:dense_amalgam} for the family
$\mathcal{W}$ will be executed in Section \ref{Section:main_result}, after various 
preparations provided in Sections \ref{Section:separation} and \ref{Section:classification_of_points}.

\section{Separation}\label{Section:separation}

In this section we study the concept of separating sets. Intuitively, a subset of a space is separating if it splits the space into two or more components. This definition however does not work for discrete spaces, such as finitely generated groups. Thus we formulate an auxiliary property called $R$-separation. In Subsection \ref{Subsection:R_separation_lemma} we analyse how 
the property of $R$-separation in groups interplays with ordinary separation
by appropriate subsets in the corresponding $E\mathcal{Z}$ structures. 
In Subsection \ref{Subsection:separation_in_group} we develop a method of creating $R$-separating sets inside a group, which correspond to subgroups, along which
the group splits algebraically. In Subsection \ref{Subsection:separation_in_EZ} we show a number of basic properties of compact separating sets in $E\mathcal{Z}$-structures. Finally, in the last Subsection \ref{Subsection:separation_in_EZ_for_graphs_of_groups} we prove a key lemma (Lemma \ref{lemma:separating_set_in_EZ}) that will allow us to easily find well-behaved compact separating sets inside $E\mathcal{Z}$-structures for groups splitting
over finite subgroups.

\subsection{\texorpdfstring{$\mathbf{R}$}{R}-separation} \label{Subsection:R_separation_lemma}

\begin{df}[Separating set]
    Let $X$ be a connected topological space and let $x_0, x_1\in X$. The set $I\subseteq X$ \textit{separates} $x_0$ from $x_1$ if $x_0$ and $x_1$ are contained in different connected components of $X\setminus I$.
\end{df}

The main idea of separation in spaces that are not necessarily connected is obtained by ``coarsifying'' the definition of separation. This is done by replacing connected components in a previous definition by $R$-path connected components (in a sense of $R$-paths described in Definition I.8.27 of \cite{bh} and already mentioned in slightly different context in the proof of Lemma \ref{limit_set_connected}). This idea is employed in a definition below.\medskip

\begin{df}[$\mathbf{R}$-separating set]\label{def:R-separating_set}
    Let $(X,d)$ be a metric space and let $R>0$ be a real number. An \textit{$R$-path} in $X$ is a finite sequence $x_0, x_1, \ldots, x_n$ of elements of $X$ such that $d(x_i,x_{i+1}) \leqslant R$ for all $i\in \set{0,1,\ldots, n-1}$. We say that a set $I \subseteq X$ \textit{$R$\nobreakdash-separates} $x_0$ from $x_1$ if $d(x_0, I) \geqslant R$, $d(x_1, I) \geqslant R$ and $x_0,x_1$ lie in two different $R$-path connected components of $X\setminus I$, i.e. there does not exist an $R$-path $x_0', x_1', \ldots, x_n'$ in $X\setminus I$ such that $x_0' = x_0, x_n' = x_1$.
\end{df}

The notions of separation and $R$-separation can be related in many ways. The following lemma, describing the way we can modify separating sets to be $R$-separating, will be useful later in this paper in the context of Cayley graphs.\medskip

\begin{lm}\label{separation-to-Rseparation}
    Let $X$ be a geodesic metric space and let $x_0,x_1$ be two points in $X$ such that some set $I$ is separating $x_0$ from $x_1$. If $d(x_0,I)\geqslant \frac{3R}{2}$ and $d(x_1,I) \geqslant \frac{3R}{2}$, then the metric neighbourhood $N_{R/2}(I)$ is $R$-separating $x_0$ from $x_1$.
\end{lm}

\begin{proof}
    First we can observe that $d(x_0,N_{R/2}(I)) \geqslant d(x_0,I) - \frac{R}{2} \geqslant R$ and $d(x_1,N_{R/2}(I)) \geqslant d(x_1,I) - \frac{R}{2} \geqslant R$, thus it is enough show that $x_0, x_1$ lie in different $R$-path connected components.

    Now suppose on the contrary that $x_0$ and $x_1$ lie in the same $R$-path connected components $X\setminus N_{R/2}(I)$, let $x_0' = x_0, x_1', \ldots, x_n'=x_1$ be an $R$-path in $X\setminus N_{R/2}(I)$. Now let $\eta_{i}:[0,1]\to X$ be a geodesic between $x_i'$ and $x_{i+1}'$. Then we can define a path $\eta: [0,n]\to X$ by setting $\eta(t) = \eta_{\floor{t}}(t-\floor{t})$ for all $t \in [0,n)$ and putting $\eta(n) = x_1$. Of course $x_0$ and $x_1$ lie in different connected components of $X\setminus I$, therefore they have to lie in different path connected components of $X\setminus I$. In particular $\eta$ is a continuous function thus we have $\eta(t)\in I$ for some $t\in [i,i+1]$ so in particular $\eta_{i}(t-i) \in I$. Since $\eta_{i}$ is a geodesic we have $$d(x_{i}',\eta_{i}(t-i)) + d(\eta_{i}(t-i),x_{i+1}') = d(x_{i}',x_{i+1}') \leqslant R$$
    and therefore we we have either $d(x_{i}',\eta_{i}(t-i))\leqslant \frac{R}{2}$ or $d(x_{i+1}',\eta_{i}(t-i))\leqslant \frac{R}{2}$. This means that $x_{i}\in N_{R/2}(I)$ or $x_{i+1}\in N_{R/2}(I)$. This contradiction ends the proof.
\end{proof}

We list below (omitting the straightforward proofs), as Lemmas 
\ref{restricting-separating-sets} and \ref{lemma:enlarging_R-separating_set},
some further basic observations concerning $R$-separation.


\medskip
\begin{lm}\label{restricting-separating-sets}
    Let $X$ be a metric space and let $x_0,x_1\in X$ be two points $R$-separated by some set $K \subseteq X$. If $X' \subseteq X$ is such that $x_0,x_1\in X'$, then $K' = K\cap X'$ is $R$-separating $x_0$ from $x_1$ in $X'$. 
\end{lm}



\medskip
\begin{lm}\label{lemma:enlarging_R-separating_set}
    Let $x_0, x_1$ be two points in a metric space $X$ and let $K$ be a set that $R$-separates $x_0$ from $x_1$. If $K'\supseteq K$ is such that $d(x_0,K')\geqslant R$ and $d(x_1,K') \geqslant R$, then $K'$ is $R$-separating $x_0$ from $x_1$.
\end{lm}



The next technical fact is a preparation for the proof of our main result in this subsection, Lemma 3.7. 

\medskip
\begin{lm} \label{lemma:continuous-change-of-covering-set}
    Let $\Gamma$ be a group acting properly and cocompactly on a 
    metric space $(X,d)$ and let $K$ be a compact set such that $\Gamma K = X$. Moreover let us define a function $f: X \to \mathscr{P}(\Gamma)$ (where $\mathscr{P}(\Gamma)$ is the family of all subsets of $\Gamma$) by
    $$f(x) = \set{\gamma\in \Gamma: x\in \gamma K}.$$
    Then for each $x\in X$ there is $\epsilon>0$ such that if $d(x,y)<\epsilon$
    then $f(x)\cap f(y)$ is nonempty.
\end{lm}

\begin{proof}
Suppose a contrario that there is a sequence $y_k$ converging to $x$
such that $f(x)\cap f(y_k)=\emptyset$ for each $k$. 
The set $\{ x,y_1,y_2,\dots \}$ is compact, hence
by properness of the action, this set is intersected only by finitely many
of the translates $\gamma K:\gamma\in\Gamma$.
As a consequence, there is $\gamma\in\Gamma$ and a subsequence
$y_{i_k}$ such that $y_{i_k}\in\gamma K$ for each $k$.
Since $\gamma K$ is a closed subset of $X$, we also have that $x\in\gamma K$, which contradicts our a contrario assumption, thus completing the proof.
\end{proof}


\begin{lm}[$\mathbf{R}$-separation lemma] \label{corollary:R-separation_lemma}
    Let $\Gamma$ be a group equipped with a word metric $d_S$ (with respect to some finite generating set $S$) and suppose that $\Gamma$ is acting properly and cocompactly on a metric 
space $X$. Moreover let $K$ be a compact set such that $\Gamma K = X$ and let $P = \set{\gamma\in \Gamma:\gamma K\cap K \neq \emptyset}$. Suppose that $\gamma_0, \gamma_1\in \Gamma$ are $R$-separated by $I\subseteq \Gamma$, where $R > \diam(P)$
(the diameter taken with respect to $d_S$). Then for any $x_0\in \gamma_0 K$ and any $x_1\in \gamma_1K$ the set $J = IK$ separates $x_0$ from $x_1$ in $X$.
\end{lm}

\begin{proof}
    First, note that $x_0,x_1\not\in J$. 
To justify this for $x_0$,
assume on the contrary that $x_0\in J$. Then $\gamma_0 K \cap J \neq \emptyset$, but this means that $\gamma_0P \cap I\neq \emptyset$, so in particular $d_S(\gamma_0, I) \leqslant \diam(P) < R$. This is a direct contradiction with our assumption. The justification for $x_1$ is the same.

    Next, define the {\it $K$-component} of $x_0$ in $X\setminus J$,
denoted $C^K(x_0,X\setminus J)$, 
to consist of all such points $y\in X\setminus J$ for which there is a 
finite sequence
$\gamma_0',\dots,\gamma_n'$ of elements of $\Gamma$
and a sequence $x_1',\dots,x_{n-1}'$ of elements of $X\setminus J$ 
such that $x_0\in\gamma_0'K$,
$y\in\gamma_n'K$, $x_i'\in\gamma_i'K$ for $i=1,\dots,n-1$, 
and
$\gamma_{i-1}'K\cap\gamma_i'K\ne\emptyset$ for $i=1,\dots,n$. 
Note that, by Lemma \ref{lemma:continuous-change-of-covering-set},
the $K$-component $C^K(x_0,X\setminus J)$ is both open and closed in 
$X\setminus J$,
and obviously it contains $x_0$. 

Now, assuming a contrario that $x_0$ and $x_1$ are in the same 
connected component of $X\setminus J$, we get that 
$x_1\in C^K(x_0,X\setminus J)$.
It follows that we can pick a sequence
$\gamma_0',\dots,\gamma_n'$  such that
$x_0\in\gamma_0'K$, $x_1\in\gamma_n'K$, 
$\gamma_i'K\cap(X\setminus J)\ne\emptyset$ for $i=0,1,\dots,n$, and
$\gamma_{i-1}'K\cap\gamma_i'K\ne\emptyset$ for $i=1,\dots,n$. 
By definition of $P$, the latter condition implies that
for $i=1,\dots,n$ we have
$\gamma_i'\in\gamma_{i-1}'P$, and hence 
$d_S(\gamma_{i-1}',\gamma_i')\leqslant\diam(P)<R$.
Moreover, for $\gamma_0,\gamma_1$ as in the assumptions,
we obviously have that $\gamma_0K\cap\gamma_0'K\ne\emptyset$,
so that $\gamma_0'\in\gamma P$, and thus 
$d_S(\gamma_0,\gamma_0')\leqslant\diam(P)<R$
(and similarly we conclude that $d_S(\gamma_n',\gamma_1)<R$). 
Consequently, we get that 
$\gamma_0,\gamma_0',\dots,\gamma_n',\gamma_1$
is an $R$-path joining $\gamma_0$ and $\gamma_1$. By our assumption, $I$ is $R$-separating $\gamma_0$ from $\gamma_1$, 
and therefore $\gamma_i'\in I$ 
for some $i\in \set{0, \ldots, n}$. 
But then $\gamma_i'K \subseteq IK = J$,
contradicting the above assured condition that 
$\gamma_i'K\cap(X\setminus J)\ne\emptyset$. Hence $x_0$ and $x_1$ lie in different connected components of $X\setminus J$, as required.
\end{proof}

\subsection{Separation in graphs of groups} \label{Subsection:separation_in_group}

In this subsection we construct some $R$-separating sets in the fundamental group $\Gamma$ of a graph of groups $\mathcal{G}$. This construction requires a word metric on $\Gamma$,
and we fix the following generating set $S$. 
For each $v\in V_Y$, let $S_v$ be any fixed finite set of generators of the vertex group $G_v$. Put
$$S = \bigcup_{v\in V_Y}S_v\cup\set{s_y:y\in O_Y\setminus O_T},$$
where $s_y$ are as in Definition \ref{def:fundamental_group_of_gog}.
Note that $S$ is obviously finite, and generates $\Gamma$. 


We start with two preparatory technical facts.
\medskip

\begin{fact} \label{fact:tiling-tree}
    Suppose that $(\mathcal{G}, Y)$ is a graph of groups with at least one edge. Let $\mathcal{T}$ be the subgraph of $\wtl{X}=\wtl{X}(\mathcal{G}, Y, T)$ with the set of edges $\abs{O}_{\mathcal{T}} = \set{\abs{G_y}:y\in O_Y}$ and the set of vertices $\set{\omega(G_y): G_y\in O_{\mathcal{T}}}$. Then such $\mathcal{T}$ is a tree. Moreover, if $S$ is defined as above, then for any $s\in S$ the intersection $s\mathcal{T}\cap \mathcal{T}$ is nonempty.
\end{fact}

\begin{proof}
    Since $\wtl{X}$ is a tree, then to show that $\mathcal{T}$ is a tree, we only need to show that it is connected. Firstly, consider the case when $T$ is a single vertex. By our assumption that $(\mathcal{G}, Y)$ has at least one edge, we take such edge $y\not\in A$ where $A$ denotes the orientation of edge as in the definition of Bass-Serre tree. Then we have $\omega(G_y) = G_{\omega(y)} = G_v$, where $v$ is the unique vertex of $Y$. In the case where $T$ is not a single vertex note that any $G_v$ is a vertex of $\mathcal{T}$ as there is an edge $y$ in $T$ ending in $v$, and therefore $\omega(G_y)=G_v$. Either way all of all vertices of form $G_v$ are vertices of $\mathcal{T}$. 
    
    Furthermore any two vertices of form $G_v,G_{v'}$ can be connected by a path in $\mathcal{T}$: let $(c_1, c_2, \ldots, c_n)$ be a path connecting $v$ and $v'$ in $T$. Then
    $$\omega(G_{c_i}) = G_{\omega(c_i)} = G_{\alpha(c_{i+1})} = G_{\omega(\ol{c_{i+1}})}=\omega(G_{\ol{c_{i+1}}}) = \alpha(G_{c_{i+1}})$$ 
    and therefore a path $(G_{c_1}, G_{c_2}, \ldots,G_{c_n})$ connects $G_v$ with $G_{v'}$. Since every edge in $\mathcal{T}$ is adjacent to at least one vertex of form $G_v$, then $\mathcal{T}$ is connected and therefore a tree.

    For the second statement we need to consider two cases:
    \begin{itemize}
        \item $s\in S_v$ for some $v\in V_Y$. Then $s\mathcal{T}\ni sG_v = G_v\in \mathcal{T}$, and therefore $s\mathcal{T}\cap\mathcal{T}\neq \emptyset$.
        \item $s = s_y$ for some $y\in O_Y\setminus O_T$. Note that the non-emptiness of $s_y\mathcal{T}\cap \mathcal{T}$ is equivalent to the non-emptiness of $s_{\ol{y}}\mathcal{T}\cap\mathcal{T} = s_y^{-1}\mathcal{T}\cap\mathcal{T}$, therefore we can assume without loss of generality that $y\in A$. Then we have $s\mathcal{T}\ni sG_{\omega(y)}\in \mathcal{T}$, and thus $s\mathcal{T}\cap\mathcal{T}\neq \emptyset$.
    \end{itemize}
    This concludes the proof of the fact.
\end{proof}


\begin{fact}\label{fact:distance_edge_and_vertex}
    Let $\gamma G_y \in O_{\wtl{X}}$ be any edge in a Bass-Serre tree $\wtl{X}$. Then for any $v\in V_Y$ the distance between $\gamma G_y$ and $\gamma G_v$ in $\wtl{X}$ is at most $\abs{\abs{O}_Y}$.
\end{fact}

\begin{proof}
    Note that both the edge $\gamma G_y$ and the vertex $\gamma G_v$ are inside the tree $\gamma \mathcal{T}$. Thus 
    $$d_{\wtl{X}}(\gamma G_y, \gamma G_v) \leqslant \abs{\abs{O}_\mathcal{T}} = \abs{\abs{O}_Y}.$$
\end{proof}

Our main result in this subsection is the following.

\medskip
\begin{lm}\label{separation-in-cayley-graph}
    Let $\Gamma=\pi_1(\mathcal{G}, Y, T)$, and let $\gamma,\gamma',\gamma''\in \Gamma$, $v,v'\in V_Y$ and $y''\in O_Y$ be such that $\gamma''G_{y''}$ is an edge on the unique geodesic path joining the vertices $\gamma G_v$ and $\gamma'G_{v'}$ in the Bass-Serre tree $\wtl{X}(\mathcal{G},Y,T)$. Then for any $R> 0$ the set $N_{R/2} (\gamma''G_{y''})$ (viewed as a subset of $(\Gamma,d_S)$) $R$-separates in $(\Gamma,d_S)$ any  $\delta\in\gamma G_v$ lying outside the set $N_{3R/2}(\gamma''G_{y''})$ from any $\delta'\in \gamma'G_{v'}$ lying outside the set $N_{3R/2}(\gamma''G_{y''})$.
\end{lm}

\begin{proof}
    We will first show that the set $\gamma''G_{y''}$ separates each pair $\delta,\delta'$ as above in the Cayley graph $\Cay(\Gamma,S)$, and then use Lemmas \ref{separation-to-Rseparation} and \ref{restricting-separating-sets} to show the $R$-separation in $(\Gamma,d_S)$.

    Assume on the contrary that $\gamma''G_{y''}$ does not separate $\delta$ from $\delta'$ in  $\Cay(\Gamma,S)$, for some $\delta,\delta'$ as above. It is immediate that $\delta, \delta'\not\in \gamma''G_{y''}$, because $d_S(\delta,\gamma''G_{y''}) > \frac{3R}{2} > 0$ and $d_S(\delta',\gamma''G_{y''}) > \frac{3R}{2} > 0$. Then there exist a path in $\Cay(\Gamma,S)$ joining $\delta$ with $\delta'$
    and omitting $\gamma''G_{y''}$. Let $\gamma_0=\delta, \gamma_1, \ldots, \gamma_n = \delta'$ be all vertices on such a path.
    
    We will now consider the union $\mathcal{U} = \bigcup_{i=0}^n \gamma_i\mathcal{T}$, where $\mathcal{T}$ is the subtree of $\wtl{X}(\mathcal{G},Y,T)$ defined in Fact \ref{fact:tiling-tree}. From the same fact we know that any two trees $\gamma_i\mathcal{T}$ and $\gamma_{i+1}\mathcal{T}$ have non-empty intersection, therefore $\mathcal{U}$ is connected. Moreover $\gamma G_v, \gamma'G_{v'}\in \mathcal{U}$, therefore $\gamma''G_{y''} \in \mathcal{U}$. From that we can conclude that $\gamma_iG_{y''} = \gamma''G_{y''}$, thus $\gamma_i\in \gamma''G_{y''}$. The contradiction ends this part of the proof.

    Now, by Lemma \ref{separation-to-Rseparation}, we know that $N_{R/2}(\gamma''G_{y''})$ $R$-separates in $\Cay(\Gamma,S)$
    any $\delta$ and $\delta'$
    as in the statement above. Moreover, by Lemma \ref{restricting-separating-sets}, we can restrict the space $\Cay(\Gamma,S)$ to $(\Gamma,d_S)$ and the set $N_{R/2}(\gamma''G_{y''})$ will still be $R$-separating $\delta$ from $\delta'$ in $(\Gamma,d_S)$. This ends the proof.
\end{proof}

\subsection{Separation by compact subsets in \texorpdfstring{$\mathbf{E}\mathbfcal{Z}$}{EZ}-structures} \label{Subsection:separation_in_EZ}

In this subsection we make a record of various useful observations  related to separation in $E\mathcal{Z}$-structures $(\ol E,Z)$ by compact subsets $K\subseteq E=\ol E\setminus Z$. We start with a series of observations
relating separation in $\ol E$ and $Z$ with that in $E$.\medskip

\begin{lm}\label{one-component}
Given an $E\mathcal{Z}$-structure $(\ol E,Z)$ and a compact subset 
$K\subseteq E$, each point $z\in Z$
belongs to the closure in $\ol E$ of precisely one connected component
of $E\setminus K$.
\end{lm}

\begin{proof}
Let $h_t:\ol E\setminus K\to\ol E\setminus K$ be a deformation as in
Remark \ref{rk:Z-set}(2). Then for a fixed $z\in Z$ the points $h_t(z)$ with $t>0$
are all in the same connected component of $E\setminus K$, say $C$,
and so $z$ belongs to the closure of $C$. If $C'$ were another connected
component of $E\setminus K$ containing $z$ in its closure, 
we would have that $C\cup C'\cup\{ z \}$ is connected. 
But this would mean that the inclusion $i:E\setminus K\to\ol E\setminus K$
maps components $C$ and $C'$ to a single connected component
of $\ol E\setminus K$, contradicting the fact that $i$ is a homotopy
equivalence (compare Remark \ref{rk:Z-set}(3)). 
\end{proof}

\begin{lm}\label{form-of-component}
Let $(\ol E,Z)$ be an $E\mathcal{Z}$-structure and let $K$ be a compact
subset of $E=\ol E\setminus Z$. Then for any connected component $C$
of $E\setminus K$:
\begin{enumerate}
\item
$C\cup\Lambda C$ is a connected component of $\ol E\setminus K$,
and each connected component of $\ol E\setminus K$ has this form;
\item
$C\cup\Lambda C$ (and thus any connected component of 
$\ol E\setminus K$) is open and closed in $\ol E\setminus K$;
\item
the limit set $\Lambda C$ is open and closed in $Z$.
\end{enumerate}
\end{lm}

\begin{proof}
To prove part 1, note first that $C\cup\Lambda C$ is obviously closed in $\ol E\setminus K$
and connected. To get the first assertion of part 1, it is then sufficient
to show that $C\cup\Lambda C$ is also open in $\ol E\setminus K$.
Let $h_t:\ol E\setminus K\to\ol E\setminus K$ be a deformation as
in Remark \ref{rk:Z-set}(2). It follows then from Lemma \ref{one-component} 
that if $t>0$ then 
$h_t(C\cup\Lambda C)\subseteq C$, and hence 
$C\cup\Lambda C=h_t^{-1}(C)$.
Since $C$ is open in $E\setminus K$ (by local connectedness of $E$),
it follows that $C\cup\Lambda C$ is indeed open in $\ol E\setminus K$,
and so the first assertion of part 1 follows. To see the second assertion
of part 1, let $D$ be a connected component of $\ol E\setminus K$.
Since the inclusion map $i:E\setminus K\to\ol E\setminus K$
is a homotopy equivalence (see Remark \ref{rk:Z-set}(3)),
the preimage $i^{-1}(D)=D\cap E$ must be connected,
and actually it must be a connected component of $E\setminus K$
since its complement in  $E\setminus K$ coincides with the
preimage by $i$ of the complement of $D$ in $\ol E\setminus K$.
Put $C=D\cap E$. We will show that $D=C\cup\Lambda C$.
The inclusion $C\cup\Lambda C\subseteq D$ follows by connectedness
of $C\cup\Lambda C$. To see the converse inclusion,
it is sufficient to show that each point of $D\cap Z$ belongs to 
$\Lambda C$. Let $z\in D\cap Z$. Then $h_t(z)\in D\cap E=C$ for
$t>0$ (by connectedness of $D$), and hence $z=h_0(z)$ is
in the closure of $C$, which completes the proof of part 1.
Part 2 follows directly from part 1 and its proof above.
Part 3 is an immmediate consequence of part 2.
\end{proof}

Lemma \ref{form-of-component} has the following three useful 
corollaries.\medskip

\begin{cor}\label{limit-clopen}
Let $(\ol E,Z)$ be an $E\mathcal{Z}$-structure, and let $K\subseteq E$
be a compact set. If $D$ is a connected component of $\ol E\setminus K$
then the intersection $D\cap Z$ is a closed and open subset of $Z$.
\end{cor}

\begin{proof}
By Lemma \ref{form-of-component}.1, 
$D=C\cup\Lambda C$ for the connected component $C$
of $E\setminus K$ coinciding with the intersection $D\cap E$. Since this
means that $D\cap Z=\Lambda C$, proposition follows by Lemma 
\ref{form-of-component}.3. 
\end{proof}

\begin{cor}\label{separation-of-two-limits}
Let $(\ol E,Z)$ be an $E\mathcal{Z}$-structure, and let 
$(e_n)$ and $(e_n')$ be two sequences of points of $E$
converging in $\ol E$ to points $z,z'\in Z$, respectively. 
Suppose that a compact subset $K\subseteq E$ separates in $E$
any point $e_i$ from any point $e_j'$. Then $K$ separates $z$
from $z'$ in $\ol E$. In particular, $z$ and $z'$ are then distinct,
and if $D$ is a connected component of $\ol E\setminus K$
containing $z$ then 
$z'\notin D\cap Z$. 
\end{cor}

\begin{proof}
Let $D$ be the connected component of $\ol E\setminus K$ which
contains $z$. By Lemma \ref{form-of-component}.2, 
$D$ is open in $\ol E\setminus K$,
so almost all points from the sequence $(e_n)$ belong to $D$,
and hence to $C=D\cap E$. Since $C$ is a connected component
of $E\setminus K$ (see Lemma \ref{form-of-component}.1), 
all points from the sequence
$(e_n')$ belong to $(E\setminus K)\setminus C$, because $K$ separates them 
in $E$ from all points $e_n$. Furthermore, since $D$ is open in 
$\ol E\setminus K$, it follows that $(\ol E\setminus K)\setminus D$
is closed in $\ol E\setminus K$, and since
$(E\setminus K)\setminus C\subseteq(\ol E\setminus K)\setminus D$,
we get that $z'=\lim e_n'$ belongs to $(\ol E\setminus K)\setminus D$.
This however means that $K$ separates $z'$ from $z$ in $\ol E$,
as required.
\end{proof}

\begin{cor}\label{separation-of-limit}
Let $(\ol E,Z)$ be an $E\mathcal{Z}$-structure, and let 
$(e_n)$ be a sequence of points of $E$
converging in $\ol E$ to a point $z\in Z$. 
Suppose that a compact subset $K\subseteq E$ separates in $E$
any point $e_i$ from some other fixed in advance point $e_0\in E$. Then $K$ separates $z$
from $e_0$ in $\ol E$.
\end{cor}

\begin{proof}
Let $D$ be the connected component of $\ol E\setminus K$ which
contains $z$. By Lemma \ref{form-of-component}.2, 
$D$ is open in $\ol E\setminus K$,
so almost all points from the sequence $(e_n)$ belong to $D$,
and hence to $C=D\cap E$. Since $C$ is a connected component
of $E\setminus K$ (see Lemma \ref{form-of-component}.1), 
it follows that 
$e_0\in(E\setminus K)\setminus C\subseteq(\ol E\setminus K)\setminus D$.
This however means that $K$ separates $e_0$ from $z$ in $\ol E$,
as required.
\end{proof}

Next few observations will be also useful in our further arguments.\medskip

\begin{lm}\label{lemma:separated_convergence_2}
    Let $(\ol{E},Z)$ be an $E\mathcal{Z}$-structure for a group $\Gamma$,
and let $\gamma_n$ be a sequence of elements of $\Gamma$ converging
to a point $z\in Z$ (which means that for some, and hence also for any
$e\in E$ we have $\lim\gamma_ne=z$).
Let $e_m$ be a sequence of points of $\ol E$ satisfying the following: 
there exists a compact subset $K\subseteq E$ and a sequence $n_k$
of natural numbers converging to infinity such that for each $k$
the set $\gamma_{n_k}K$ separates (in $\ol E$) the element $e_k$
from almost all elements $e_m$. Then $\lim e_m=z$.
\end{lm}

\begin{proof}
Suppose a contrario that $e_m$ does not converge to $z$.
By compactness of $\ol E$, there is a subsequence $e_{m_i}$
converging to some $e_0\in\ol E\setminus\{ z \}$.
Let $U$ be an open and connected neighbourhood of $e_0$ in $\ol E$
not containing $z$ in its closure (such $U$ exists due to local
connectedness of $\ol E$). 
Then for all sufficiently large $k$ the sets $\gamma_{n_k}K$
are disjoint with $U$, and hence $U$ is contained in a single connected
component of $\ol E\setminus\gamma_{n_k}K$. Since almost all $e_{m_i}$
are contained in $U$ (because $e_{m_i}\to e_0\in U$),
this contradicts the final assumption of the lemma,
thus completing the proof.
\end{proof}

\begin{lm}\label{translates-in-single-component}
Let $(\ol{E},Z)$ be an $E\mathcal{Z}$-structure for a group $\Gamma$, and
let $K,L$ be any compact subsets of $E$. Then for almost all 
$\gamma\in\Gamma$
the translate $\gamma L$ is disjoint with $K$ and contained in a single connected component of $E\setminus K$. As a consequence, 
for almost all $\gamma\in\Gamma$ the set $L$ is contained in a single
connected component of $E\setminus \gamma K$. 
\end{lm}

\begin{proof}
To prove the lemma, we need the following two preparatory observations.

{\bf Claim 1.}
{\it For any two points $p,q$ of $E$ there is a compact connected subset of $E$
containing these two points.}

To prove the claim, 
consider the union $P$ of all compact connected subsets of $E$ 
containing $p$.
By local connectedness and local compactness of $E$
(which follow from the fact that $E$ is an open subset of $\ol E$), 
$P$ is both open
and closed in $E$. Then, by connectedness of $E$, we get $P=E$,
hence the claim. 

{\bf Claim 2.}
{\it For any compact set $K\subseteq E$ there is a compact connected
subset $K'\subseteq E$ containing $K$.}

To prove Claim 2, note that
by local connectedness and local compactness of $E$, we can cover $K$
by a finite collection $\mathcal{N}$ of compact connected sets. Then,
by referring to Claim 1, for any two sets from
$\mathcal{N}$ choose a compact connected set which intersects them
both, getting another finite family $\mathcal{M}$ of compact connected sets.
The union $K'\coloneqq\bigcup(\mathcal{N}\cup\mathcal{M})$ is then obviously
compact and connected, and contains $K$.

We now pass to the essential part of the proof of Lemma 
\ref{translates-in-single-component}.
Let $L'$ be a compact connected subset of $E$ which contains $L$
(and which exists due to Claim 2.
By properness of the action of $\Gamma$ on $E$,
for almost all $\gamma\in\Gamma$ the translate $\gamma L'$ is disjoint
with $K$, and since it is connected, it must be contained in a single
connected component of $E\setminus K$. Our first assertion of the lemma 
holds then
for the same $\gamma\in\Gamma$ 
by the inclusions $\gamma L\subseteq \gamma L'$.
The second assertion of the lemma follows from the first one by 
$\Gamma$-invariance.
\end{proof}

\subsection{Separation in \texorpdfstring{$\mathbf{E}\mathbfcal{Z}$}{EZ}-structures for graphs of groups with finite edge groups} \label{Subsection:separation_in_EZ_for_graphs_of_groups}

In this subsection we make some preparatory observations concerning
separation by compact subsets in $E\mathcal{Z}$-structures for
fundamental groups of graphs of groups with finite edge groups.
More precisely, we relate separation properties of edges in the corresponding
Bass-Serre trees with separation properties of appropriate compact subsets
in the corresponding $E\mathcal{Z}$-structures.


\medskip

\begin{lm} \label{lemma:separating_set_in_EZ}
    Let $(\mathcal G,Y)$ be a graph of groups with finite edge groups,
    let $\Gamma$ be its fundamental group, and let $\wtl{X}$ be the corresponding
    Bass-Serre tree. Let $(\ol{E},Z)$ be an $E\mathcal{Z}$-structure for $\Gamma$, and
    let $e_0\in E$ be any fixed point. Then, there exist a compact set $K\subseteq E$ containing $e_0$ and a constant $R_0>0$ such that the following holds.
    Let $\gamma'' G_{y''}$ be any edge of $\wtl{X}$, and let
    $\wtl{X}_0, \wtl{X}_1$ denote the two connected components obtained by deleting 
    the interior of  ${\gamma''G_{y''}}$ from $\wtl{X}$.
    Finally, define the subsets $M,M'\subseteq\Gamma$ by
    $$M =\bigcup_{\gamma G_v \in V_{\wtl{X}_0}}\gamma G_v,$$
    $$M'=\bigcup_{\gamma' G_{v'} \in V_{\wtl{X}_1}}\gamma'G_{v'}.$$
    Then the set $\gamma'' K$ separates in $E$ any  point $\delta e_0$ with $\delta\in M\setminus N_{R_0}(\gamma''G_{y''})$ from any point $\delta'e_0$ with $\delta'\in M'\setminus N_{R_0}(\gamma''G_{y''})$.
    Consequently, the set $\gamma'' K$ separates in $E$ all but finitely many points $\delta e_0 \in Me_0$ from all but finitely many points $\delta'e_0\in M'e_0$.
\end{lm}

\begin{proof}
    By cocompactness and properness of the group action, we know that there exists 
    a compact subset $L\subset E$ such that $\Gamma L = E$ and $e_0\in L$. We can now take a finite symmetric set $P = \set{\gamma\in\Gamma:\gamma L\cap L \neq \emptyset}$. Since $P$ is finite by properness of action of $\Gamma$ on $E$, then it also has a finite diameter in $\Gamma$ equipped with the metric $d_S$. We define
    $$I_r =N_{r}\left(\bigcup_{y\in\abs{O}_Y}G_y\right)$$
    and put
    $K = I_{\diam(P)/2}L$. Firstly observe that the set  $I_{\diam(P)/2}$ is finite since the generating set $S$ is finite and all of the groups edge groups $G_y$ are finite. Therefore $K$ is a union of finitely many compact sets, and thus a compact set. We will show that such chosen set $K$ is as required (and a required constant $R_0$ will be specified later). Similarly to the argument for compactness of the set $K$, we note that the set $I_{3\diam(P)/2}$ is finite. By Lemma \ref{separation-in-cayley-graph} we know that the set $N_{\diam(P)/2}(\gamma''G_{y''})$ is $\diam(P)$-separating all the points $\delta\in M\setminus N_{3\diam(P)/2}(\gamma'' G_{y''})$ from all of the points $\delta'\in M'\setminus N_{3\diam(P)/2}(\gamma'' G_{y''})$. In particular, $N_{\diam(P)/2}(\gamma''G_{y''})$ is also $\diam(P)$\Hyphdash*separating all of the points 
    $\delta\in M \setminus \gamma''I_{3\diam(P)/2}$
    from all of the points 
    $\delta'\in M' \setminus \gamma''I_{3\diam(P)/2}.$
    Since $N_{\diam(P)/2}(\gamma''G_{y''}) \subseteq I_{\diam(P)/2}$ and for any $\delta, \delta'$ as above we have 
    $d_S\left(\delta, \gamma''I_{\diam(P)/2}\right) \geqslant \diam(P)$
    and 
    $d_S\left(\delta', \gamma''I_{\diam(P)/2}\right) \geqslant \diam(P),$
    then by Lemma \ref{lemma:enlarging_R-separating_set} we know that $\gamma''I_{\diam(P)/2}$ is also $\diam(P)$-separating $\delta$ from $\delta'$. By the $R$-separation lemma (Lemma \ref{corollary:R-separation_lemma}) the set $\gamma''K$ separates $\delta e_0$ from $\delta'e_0$,
    and the lemma holds for $R_0$ equal to the diameter in $(\Gamma,d_S)$ of the set $I_{3\diam(P)/2}$. 
\end{proof}

As fairly direct corollaries of the above lemma we get the following two useful observations (the justifications of which we omit). 

\medskip

\begin{fact}\label{lemma:final_separation}
Under the assumptions as in Lemma \ref{lemma:separating_set_in_EZ}, 
let $K\subseteq E$ be as in the assertions of this lemma. Then for any two vertices $\gamma G_v$, $\gamma'G_{v'}$ of the Bass-Serre tree $\wtl{X}$ and any edge $\gamma''G_{y''}$ lying on the geodesic path in $\wtl{X}$ between those vertices, the set $\gamma''K$ separates 
in $E$ all but finitely many $\delta e_0$, with $\delta\in \gamma G_v$, from all but finitely many $\delta'e_0$, with $\delta'\in \gamma' G_{v'}$. 
\end{fact}

\medskip

\begin{fact}\label{fact:sufficiently_far_vertices}
With the assumptions as above, there exist a constant $R>0$ such that if the distances 
in $\wtl{X}$ between the edge $\gamma''G_{y''}$ and the vertices $\gamma G_v$, $\gamma'G_{v'}$ are both greater than $R$, then $\gamma''K$ separates in $E$ all $\delta e_0$ as above from all $\delta'e_0$ as above.
\end{fact}

\section{Points in \texorpdfstring{$\mathbf{Z}$}{Z}}\label{Section:classification_of_points}

The standing assumption for this section is that $\mathcal{G}$
is a graph of groups with finite edge groups, 
$\Gamma=\pi_1\mathcal{G}$ is its fundamental group,
and  $(\ol{E},Z)$ is an $E\mathcal{Z}$-structure for $\Gamma$.
We proceed to study the topological structure of the set $Z$.
In Subsection \ref{Subsection:categorisation_of_points}
we describe two kinds of points 
in $Z$, which we call {\it branch points} and {\it vertex points} (see Definitions \ref{def:infinite_branch_point} and \ref{def:vertex_point}, respectively), and we show that there are no other points in $Z$. In Subsection \ref{properties_of_points_in_Z} we study the properties of aforementioned points both as limits of points from $E$ and by stating their topological properties inside the space $Z$. In Subsection \ref{Subsection:boundary_of_virt_free_groups} we use the previous two sections in order to describe the space $Z$ in case when all of the vertex groups $G_v$ are finite (so that there are no vertex points in $Z$). In that case the studied group $\Gamma$ is finite, virtually $\ZZ$ or virtually free,
and we get that the topological shape of $Z$ is uniquely determined in each of these three cases (i.e. $Z$ is empty,
a doubleton or it is homeomorphic to the Cantor space, respectively). Even though these latter observations seem to belong to the ``mathematical folklore'', we have decided to include the analysis of those cases for the sake of completeness of this paper.

\subsection{Classification of points in \texorpdfstring{$\mathbf{Z}$}{Z}}\label{Subsection:categorisation_of_points}

Let $\wtl{X}$ be the Bass-Serre tree of $\mathcal{G}$.

\medskip
\begin{df}\label{Def:phi}
    Let $\varphi\colon O_{\wtl{X}}\to \Gamma$ be a function such that $\varphi(\gamma G_y) \in \gamma G_y$. 
\end{df}

Even though the function $\varphi$ is not uniquely determined, it is quite well-behaved, due to finiteness of the edge groups in $(\mathcal{G},Y)$. The basic properties of $\varphi$ are described in the next two facts, for which we omit the straightforward proofs.\medskip

\begin{fact}\label{fact:phi_phi'_close}
    There exists a constant $D$ such that for any two functions $\varphi, \varphi'$ as above, and for any edge $\gamma G_y$ of $\wtl{X}$, we have $d_S(\varphi(\gamma G_y), \varphi'(\gamma G_y)) \leqslant D$.
\end{fact}


\medskip
\begin{fact}\label{fact:phi_finite_to_one}
    The function $\varphi$ as above is uniformly finite-to-one.
\end{fact}


\com{\medskip

\begin{lm}\label{lemma:phi_lies_close_in_BS_tree}
    Let $c = \gamma G_y\in O_{\wtl{X}}$ be any oriented edge. Then we have the following estimate for the distance in the Bass-Serre tree:
    $$d_{\wtl{X}}(c, \varphi(c)G_{\omega(y)}) \leqslant 2\abs{\abs{O}_Y}.$$
\end{lm}

\begin{proof}
    We need to consider two cases. If $c\in A\setminus O_T$, then we have
    $$\omega(c) = \omega(\gamma G_y) = \omega(\varphi(c)G_y)  = \varphi(c) s_y G_{\omega(y)} \in \varphi(c)s_y \mathcal{T}.$$
    On the other hand we have $\varphi(c)G_{\omega(y)} \in \varphi(c) \mathcal{T}$. Since trees $\varphi(c)s_y \mathcal{T}$ and $\varphi(c) \mathcal{T}$ have a nonempty intersection by Fact \ref{fact:tiling-tree}, then we can estimate
    $$d_{\wtl{X}}(c, \varphi(c)G_{\omega(y)}) \leqslant 2\diam(\mathcal{T}) \leqslant 2\abs{\abs{O}_\mathcal{T}} = {\abs{O}_Y}.$$
    
    In the other case, we know $\omega(c) = \gamma G_{\omega(c)}$. Since $G_y < G_{\omega(y)}$ in that case we conclude that $\omega(c) = \gamma g G_{\omega(c)} = \varphi(c)G_{\omega(y)}$, thus the edge $c$ is adjacent to $\varphi(c)G_{\omega(y)}$.
\end{proof}}

Using the function $\varphi$, to each branch in $\wtl{X}$ we can associate a sequence of points in $E$. The following lemma shows, that such sequence is always convergent to a point in $Z$.\medskip

\begin{lm}\label{lemma:reduced_word_converge_new}
    Fix a base point $e_0\in E$. For any branch $c=(c_1,c_2,\dots)$ in $\wtl{X}$, the sequence of points $\varphi(c_n)e_0$ does converge to a point $z\in Z$.

\end{lm}

\begin{proof}
    Put $e_n = \varphi(c_n)e_0$ and $c_n = \gamma_n'G_{y_n}$. Of course, since $\ol{E}$ is compact, we can choose a convergent subsequence $e_{n_k}$. By properness of the action of $\Gamma$ on $E$,
    the limit $z = \lim_{k\to\8}e_{k}$ is a point of $Z$. Now let us take a compact set $K \subseteq E$ as in the assertion of Lemma \ref{lemma:separating_set_in_EZ} and a positive integer $R$ as in the assertion of Fact \ref{fact:sufficiently_far_vertices}. We will show that a sequences $\gamma_k = \varphi(c_{n_k})$ and $e_n$ satisfy the conditions of Lemma \ref{lemma:separated_convergence_2}.

    For any given $l\in \NN$ let $k_l$ be a natural number such that  $n_{k_l} > l + R$. Of course such chosen sequence of natural numbers $n_{k_l}$ converges to infinity. It remains to show that a compact set $\gamma_{n_{k_l}}K$ separates the point $e_l$ from all but finitely many $e_n$ in $\ol{E}$.

    Since $d_{\wtl{X}}(c_l,c_{n_{k_l}}) > R$, then $d_{\wtl{X}}(\alpha(c_l),c_{n_{k_l}}) > R$ and $d_{\wtl{X}}(\omega(c_l),c_{n_{k_l}}) > R$. Moreover for any $n > n_{k_l} + R$ we also have $d_{\wtl{X}}(c_n,c_{n_{k_l}}) > R$, then in particular $d_{\wtl{X}}(\alpha(c_n),c_{n_{k_l}}) > R$ and $d_{\wtl{X}}(\omega(c_n),c_{n_{k_l}}) > R$. In particular, any point of form $\delta e_0$ where $\delta$ belongs to the union of cosets $\alpha(c_l)\cup \omega(c_l)$ is separated from any point of form $\delta' e_0$ where $\delta'$ belongs to the union of cosets $\alpha(c_n)\cup \omega(c_n)$. By remark \ref{remark:containment_of_edge} $\varphi(c_l) \in c_l \subseteq \alpha(c_l) \cup \omega(c_l)$ and $\varphi(c_n) \in c_n \subseteq \alpha(c_n) \cup \omega(c_n)$, then $\gamma_{n_{k_l}}K$ separates $e_l$ from all but finitely many $e_n$. By Lemma \ref{lemma:separated_convergence_2} the sequence $e_n$ converges.
\end{proof}

Moreover it turns out that the limit of a sequence associated to the branch does only depend on the branch itself. This observation is formulated in a lemma below.\medskip

\begin{lm}
    The limit $\lim_{n\to \8} \varphi(c_n)e_0$ as in the previous lemma does not depend on the initial choice of the function $\varphi$.
\end{lm}

\begin{proof}
    Let $\varphi, \varphi'$ be any functions as in Definition \ref{Def:phi}, and let $D$ be a constant from Fact \ref{fact:phi_phi'_close}. Then we can take a set $K_n = \varphi(c_n)B_S(1,D)e_0$. By this definition and Fact \ref{fact:phi_phi'_close} we have that $\varphi(c_n)e_0, \varphi'(c_n)e_0\in K_n$. The sets $K_n$ are all $\Gamma$-translates of a finite (and thus compact) set $K_0 = B_S(1,D)e_0$. Since the family of all $\Gamma$-translates of $K_0$ is null by the definition of an $E\mathcal{Z}$-structure, the limits of the sequences $\varphi(c_n)e_0$ and $\varphi'(c_n)e_0$ have to be equal.
\end{proof}

The two previous lemmas motivate our definition of a branch point.\medskip

\begin{df}\label{def:infinite_branch_point}
    We say that a point $z \in Z$ is an \textit{branch point} if it can be expressed as a limit 
    $$z = \lim_{n\to\8}\varphi(c_n)\cdot e_0$$
    for some branch $c = (c_1, c_2, \ldots)$ in $\wtl{X}$.

    Note that in the definition above we can require without loss of generality that all involved branches start at a fixed base vertex of the Bass-Serre tree. This does not affect the set of branch points since for any vertex $\gamma G_v$ and any branch $c$ in $\wtl{X}$ there exists a branch $c'$ starting at $\gamma G_v$ which is cofinal with $c$ (i.e. coincides with $c$
    except at finitely many initial edges).
\end{df}

It turns out that for any two different branches starting at the same vertex of $\wtl{X}$ their associated points in $Z$ have to be different.\medskip

\begin{lm}\label{lemma:different_branch_points}
    Suppose that $c = (c_1, c_2, \ldots), c' = (c_1', c_2', \ldots)$ are two different branches in a Bass-Serre tree $\wtl{X}$ starting at the same base vertex $\gamma_0G_{v_0}$. Then
    $$\lim_{n\to\8}\varphi(c_n)e_0\neq \lim_{n\to\8}\varphi(c_n')e_0.$$
\end{lm}

\begin{proof}
    Let $m\in \NN$ be such that $c_m \neq c_m'$. Then for any $n,n' > m$ the final vertices $\omega(c_n)$ and $\omega(c_{n'}')$ of the corresponding oriented edges
    lie in distinct connected components of the graph obtained from $\wtl{X}$
    by deleting the interior of the edge $|c_m|$. 
    Moreover, by the construction of the Bass-Serre tree $\wtl{X}$,
    we know that if $c=\gamma G_y$, $\omega(c)=\gamma' G_{v'}$ and $\alpha(c)=\gamma'' G_{v''}$,
    then $\gamma G_y\subseteq\gamma' G_{v'}\cup \gamma'' G_{v''}$ (compare Remark \ref{remark:containment_of_edge}), and hence $\varphi(c)\in\gamma' G_{v'}\cup \gamma'' G_{v''}$.
    Combining the observations from the last two sentences, and applying
    Lemma \ref{lemma:separating_set_in_EZ}, we get that there exists a compact set $K\subseteq E$ which separates all but finitely many elements from the sequence $\varphi(c_n)e_0$ from all but finitely many elements from the sequence $\varphi(c_n')e_0$. Therefore, by Corollary \ref{separation-of-two-limits}, we have
    $$\lim_{n\to \8}\varphi(c_n)e_0 \neq \lim_{n\to \8}\varphi(c_n')e_0.$$
\end{proof}

We now turn to describing the {\it vertex points} of $Z$.
\medskip

\begin{df}\label{def:vertex_point}
    Let $G_v$ be a vertex group of $\mathcal{G}$. We say that a point $z\in Z$ is a \textit{$G_v$-vertex point} if it can be expressed as a limit
    $$z = \lim_{n\to\8}\gamma_n\cdot e_0,$$
    where $\gamma_n \in \Gamma$ is a sequence of group elements from some coset $\gamma G_v$. We also abbreviate and say that $z$ is a \textit{vertex point} if it is a $G_v$-vertex point for some vertex $v\in V_Y$.
\end{df}

Our aim now is to show that $Z$ splits as the disjoint union of the sets of its branch points and its vertex points.
This "classification" of the points of $Z$ will be justified by next three observations.
\medskip

\begin{lm}\label{lemma:vertex_and_branch_points_are_disjoint}
The sets of all vertex points in $Z$ and all branch points in $Z$ are disjoint.
\end{lm}

\begin{proof}
    Let $c = (c_1, c_2, \ldots)$ be an branch in $\wtl{X}$ and let $\gamma_n$ be a sequence of elements in a coset $\gamma G_v$. Suppose that 
    \begin{align*}
        \lim_{n\to\8}\varphi(c_n)e_0 = z &&\text{and} && \lim_{n\to\8} \gamma_ne_0 = z',
    \end{align*}
    we will show that $z\neq z'$. Let $n\in \NN$ be such that $\alpha(c_n)$ is the closest vertex to $\gamma G_v$ in $\wtl{X}$ and let $\wtl{X}_0$ denote the connected component of $\wtl{X}$ with the interior of $c_n$ removed containing $\omega(c_n)$. Then, since $\wtl{X}$ is a tree, for all $m> n$ vertices $\alpha(c_m)$ and $\omega(c_m)$ are contained in $\wtl{X}_0$. Note that by Remark \ref{remark:containment_of_edge} this consequently means that $\varphi(c_m)\in \bigcup V_{\wtl{X}_0}$. If $c_n = \gamma'G_y$, then by Lemma \ref{lemma:separating_set_in_EZ} we know that there exist a compact set $K\subseteq E$ such that $\gamma'K$ is separating all but finitely many points of form $\varphi(c_m)e_0$ from all but finitely many points $\gamma_ne_0$. In particular by using Corollary \ref{separation-of-two-limits} we conclude that
    $$z = \lim_{n\to\8}\varphi(c_n)e_0 \neq \lim_{n\to\8} \gamma_ne_0 = z',$$
    which ends the proof.
\end{proof}

To complete our classification of the points of $Z$ we will need the following auxiliary technical fact.\medskip

\begin{fact}\label{fact:close_enough_to_Gv}
    Let $z = \lim_{n\to \8} \gamma_n e_0$ be a point in $Z$. If there exists a positive constant $r$ such that $d_S(\gamma G_v,\gamma_n) \leqslant r$ for all $n$, then $z$ is a $G_v$-vertex point.
\end{fact}

\begin{proof}
    Let $K \subseteq E$ be a compact set such that $\Gamma K = E$ and $e_0\in K$. 
    Let $\ol{B}_{d_S}(1, r)\subseteq\Gamma$ be the ball around the unit $1\in\Gamma$,
    with respect to the word metric $d_S$.
    Put $K' = \ol{B}_{d_S}(1, r)K$, and note that it is a compact set containing $K$
    (and thus also containing $e_0$). 
    Furthermore, for each $n$ pick $\gamma_n'\in \gamma G_v$ such that $d_S(\gamma_n, \gamma_n') \leqslant r$.
    Then for each $n$ we have $\gamma_n e_0\in\gamma_n'\ol{B}_{d_S}(1, r)e_0\subseteq\gamma_n'K'$. It follows that for any $n$ both $\gamma_ne_0\in\gamma_n'K'$ and
    $\gamma_n'e_0\in\gamma_n'K'$.
    Since the diameters of $\gamma_n'K'$ converge to 0, we get that
    $\lim\gamma_n'e_0=z$, which completes the proof.
\end{proof}


\begin{lm}\label{lemma:classification_in_Z}
    Each point $z \in Z$ is a branch point or a $G_v$-vertex point for some vertex group $G_v$.
\end{lm}

\begin{proof}
    Suppose that $z$ is any point in $Z$. Let $e_0\in K$ be any fixed base point, then $z$ can be expressed as a limit $\lim_{n\to\8} \gamma_ne_0$. Let now $K$ be a compact set and $R_0$ a constant as in the assertion of Lemma \ref{lemma:separating_set_in_EZ}.
    
    Fix any $y\in O_Y$ and define a function $\psi_y:\Gamma\to O_{\wtl{X}}$ as $\psi_y(\gamma) = \gamma G_y$. Consider the minimal subtree $\Upsilon$ of $\wtl{X}$ containing all of the edges $\psi_y(\gamma_n)$ for $n\geqslant1$. Since the sequence $\gamma_n$ has infinitely many distinct values and the function $\psi_y$ is finite-to-one, we know that $\Upsilon$ contains infinitely many edges and therefore it is infinite. By K\H{o}nig's lemma, either there is a vertex of infinite degree in $\Upsilon$, or $\Upsilon$ contains a branch. 
    
    Suppose first that $\gamma G_v$ is a vertex of $\Upsilon$ with infinite degree. Then there is a sequence $\gamma g_n G_{y_n}$ of pairwise distinct edges in $\Upsilon$ such that $\omega(\gamma g_n G_{y_n}) = \gamma G_v$. 
    Without loss of generality, by referring to compactness of $\ol{E}$, we can assume that the sequence $g_n$ is chosen in such a way that $\gamma g_ne_0$ is convergent, and by properness of the $\Gamma$-action on $E$, this sequence converges then to some point $z'\in Z$. 
    By the definition of a final vertex of an edge in a Bass-Serre tree, we know that for each $n$ we have either $g_n\in G_v$ or $g_n = g_n's_{y_n}^{-1}$ for $g_n'\in G_v$. Either way this means that 
    $$\label{eq:4.11.1}
    d_S(\gamma G_v, \gamma g_n) \leqslant 1.  \leqno{(4.11.1)}
    $$
    Therefore, by Fact \ref{fact:close_enough_to_Gv}, the point $z'$ is a $G_v$-vertex point.
    We will show that $z$ is also a $G_v$-vertex point in this case.
    
    Let $\gamma_n'$ be such a subsequence of $\gamma_n$ that $\gamma_n'G_y$ is an edge whose interior is contained in the same connected component of $\Upsilon\setminus \set{\gamma G_v}$ as the interior of $\gamma g_nG_{y_n}$.
    Consider the subcase $n_k$ which for infinitely many $k\in \NN$ we have
    $\gamma_{n_k}'\in N_{R_0}(\gamma g_{n_k} G_{y_{n_k}})$. Then, by the fact that $1\in G_{y_n}$, we get that 
    $$
    d_S(\gamma_n',\gamma g_n)\leqslant R_0+\diam(G_{y_n})\leqslant
    R_0+\max[\diam(G_y):y\in O_Y]
    $$
    for all these $n$. Consequently, by applying the above estimate (\hyperref[eq:4.11.1]{4.11.1}), for all these $n$ we have
    $$
    d_S(\gamma_n',\gamma G_v)\leqslant R_0+\max[\diam(G_y):y\in O_Y]+1.
    $$
    In view of the latter, forming an increasing sequence $n_k$ out of the above $n$,
    and applying again Fact \ref{fact:close_enough_to_Gv},
    we get that 
    $$z = \lim_{n\to \8} \gamma_ne_0 = \lim_{n\to \8} \gamma_n'e_0 =
    \lim_{k\to\8}\gamma_{n_k}'e_0 = \lim_{k\to\8}\gamma g_{n_k}e_0 = z',
    $$
    which proves that $z$ is a $G_v$-vertex point.

    Now, if we are not in the subcase considered above,
    we can assume without loss of generality that for all $n$
    we have $\gamma_n'\notin N_{R_0}(\gamma g_n G_{y_n})$.
    In this subcase we rely on the observation that the interior of the edge
    $\gamma g_nG_{y_n}$ separates in $\wtl X$ at least one vertex, say $\bar\gamma_nG_{v_n}$,
    of the edge $\gamma_n'G_y$, from all edges $\gamma_m'G_y$ where $m>n$, and in particular
    from all initial vertices $\alpha(\gamma_m'G_y)\eqqcolon\delta_m'G_{v_m'}$ and all final vertices $\omega(\gamma_m'G_y)\eqqcolon\bar{\delta}_m'G_{\bar{v}_m'}$ of these edges.
    By construction of the Bass-Serre tree, we know that $\gamma_n'\in\delta'_nG_{v_n'}\cup\bar{\delta}'_nG_{\bar{v}_n'}$
    and that $\gamma_m'\in\delta'_mG_{v_m'}\cup\bar{\delta}'_mG_{\bar{v}_m'}$ for all $m>n$ (compare Remark \ref{remark:containment_of_edge}).
    It follows then from Lemma \ref{lemma:separating_set_in_EZ} that the set $\gamma g_n K$ separates in $E$
    the point $\gamma_n'e_0$ from all but finitely many of the points $\gamma_m'e_0$ for $m>n$ (here we use the fact that 
    $\gamma_n'\notin N_{R_0}(\gamma g_nG_{y_n})$).
    This means however that the sequences $\gamma g_n$ (as substituted for $\gamma_n$ from the statement of Lemma 3.16)
    and $e_n\coloneqq\gamma_n'e_0$ satisfy the assumptions of Lemma \ref{lemma:separated_convergence_2}. Therefore
    $$z' = \lim_{n\to\8} \gamma g_n e_0 = \lim_{n\to\8} \gamma'_n e_0 = z,$$
    which again proves that $z$ is a $G_v$-vertex point.

    Suppose now, that there is a branch $c=(c_1,c_2,\dots)$ in 
    $\Upsilon$, and let $c_n =g_n G_{y_n}$. By Lemma \ref{lemma:reduced_word_converge_new}, the sequence $g_n e_0$ converges then to some point $z'\in Z$ which is a branch point. Let $K$ be as in the assertion of Lemma \ref{lemma:separating_set_in_EZ},
    and let additionally $R>0$ be as in the assertion of Fact \ref{fact:sufficiently_far_vertices}. Recalling that $\Upsilon$ is spanned by the edges 
    $\gamma_n G_y$, it is not hard to construct recursively subsequences $c_{i_n}$ of $c_n$, and $\gamma_{j_n}$ of $\gamma_n$, such that for each $n$:
    \item{(1)} the edge $c_{i_n}$ separates in $\wtl{X}$ the edge $\gamma_{j_n}G_y$ from all of the edges $\gamma_{j_m}G_y$ for $m>n$, and
    \item(2) the distances in $\wtl{X}$ from $c_{i_n}$ to any of the edges $\gamma_{j_m}G_y$ for $m\geqslant n$ are all greater than $R$.

    By the reasoning as in the last paragraph of the arguments for the previous case, and in view of Fact \ref{fact:sufficiently_far_vertices}, we get that for each $n$ the set $g_{i_n}K$ separates in $E$ the point $\gamma_{j_n}e_0$ from any of the points $\gamma_{j_m}e_0$ for $m>n$. But this means that the sequences $g_{i_n}$
    (substituted for $\gamma_n$ from the statement of Lemma 3.16)
    and $e_n\coloneqq\gamma_{j_n}e_0$ satisfy the assumptions of Lemma 3.16. Therefore, since $z'=\lim g_{i_n}e_0$, we get that $z=\lim\gamma_{j_n}e_0=z'$, and hence $z$ is a branch point.
\end{proof}

Combining Lemmas \ref{lemma:vertex_and_branch_points_are_disjoint} and \ref{lemma:classification_in_Z} we get the following.\medskip

\begin{cor}[Classification of points in $\mathbf{Z}$]\label{cor}
    The sets of branch points and of vertex points are disjoint and cover the entire $Z$.
\end{cor}

\subsection{Some topological properties of \texorpdfstring{$\mathbf{Z}$}{Z}}\label{properties_of_points_in_Z}

In this subsection, apart from keeping the standing assumptions for all of Section \ref{Section:classification_of_points}, we assume additionally that graphs of groups under consideration are non-elementary (as in Definition \ref{Definition:non-elementary_gog}). We aim to prove two lemmas regarding the topology of $Z$ (Lemmas \ref{lemma:branch_points_dense} and \ref{lemma:H_W-saturated} below). Both of those lemmas are crucial for our proof of Theorem A provided in the next section. For the first of these lemmas however we need the observation stated below.\medskip

\begin{lm}
    If $v\in V_Y$ is a vertex such that the group $G_v$ is infinite, then there exists an edge $y\in O_Y$ with $\alpha(y)=v$ such that for any $\gamma\in \Gamma$, removing the interior of the edge $\gamma G_y$ splits the Bass-Serre tree $\wtl{X}$ into two trees of infinite diameter.
\end{lm}

\begin{proof}
    By Lemma 4.1.6 (1) in \cite{dense-amalgam}, there exist an edge $y$ with $\alpha(y)=v$ any lift of which separates $\wtl{X}$ into two subtrees, each containing lifts of all vertices of $Y$. We will show that each of those subtrees has infinite diameter. For that, let us fix a lift $\gamma G_y$ of $y$. Suppose on the contrary that one of those trees, say $\wtl{X}_0$, has finite diameter, and let $\gamma'G_v$ be a lift of $v$ contained in $\wtl{X}_0$ for which $d_{\wtl{X}}(\gamma G_y, \gamma'G_v)$ is maximal. Then, since $G_v$ is infinite and $G_y$ is finite, there exist infinitely many $\gamma''\in\Gamma$ such that $\alpha(\gamma''G_y) = \gamma'G_v$. In particular we can take $\gamma''$ as above
    such that $\gamma''G_y$ is not contained in the geodesic between $\gamma'G_v$ and $\alpha(\gamma G_y)$. 
    Let $\wtl{X}''_1$ be the connected component of 
    $\wtl{X}\setminus \hbox{int}(\gamma''G_y)$ not containing
    the edge $\gamma G_y$.
    Then $\wtl{X}''_1$ is contained in $\wtl{X}_0$, and any vertex of  $\wtl{X}''_1$ lies at distance at least $d_{\wtl{X}}(\gamma G_y,\gamma'G_v) + 1$ from $\gamma G_y$. Moreover, by referring again to Lemma 4.1.6 (1) from \cite{dense-amalgam}, $\wtl{X}''_1$ contains at least one lift of $v$, and thus we have found a lift of $v$ in $\wtl{X}_0$ that lies further from $\gamma G_y$ than $\gamma'G_v$. The contradiction ends the proof.
\end{proof}

\medskip

\begin{lm}\label{lemma:branch_points_dense}
    The set of all branch points is dense in $Z$.
\end{lm}

\begin{proof}
    By the classification of points in $Z$ given in Corollary \ref{cor}, it is enough show a sequence of branch points converging to any vertex point. Let $z = \lim_{n\to \infty}\gamma g_ne_0$ (where $g_n\in G_v$) be a vertex point. In that case group $G_v$ has to be infinite, let $y\in O_Y$ be an edge such that $\omega(y)=v$ and any lift of $y$ splits $\wtl{X}$ into two trees with infinite diameters (existence of such $y$ is assured by the previous lemma). Now, if $y\in A\setminus O_T$ we define $g_n' = g_ns_y$, and otherwise define $g_n' = g_n$. Note that then the final vertex of each edge of form $\gamma g_n'G_y$ is $\gamma G_v$. Since the sequence $g_n$ has infinitely many distinct values, the same holds for $g_n'$. Combining this with the fact that each coset of $G_y$ is finite, we can without loss of generality (by passing to a subsequence) assume that $\gamma g_n' G_y\neq \gamma g_{n'}'G_y$ for $n\neq n'$. For each $n$, let $\wtl{X}^{(n)}$ be the component of $\wtl{X}\setminus\hbox{int}(\gamma g_n'G_y)$ not containing the vertex $\gamma G_v$, let $c^{(n)}=(c_m^{(n)})$ be a branch in $\wtl{X}^{(n)}$ and let $z^{(n)} = \lim_{m\to \8}\varphi(c^{(n)}_m)e_0$ be the corresponding branch point. Note that, since $\gamma g_n'G_y$ and $\gamma g_{n'}'G_y$ are two edges sharing the vertex $\gamma G_v$, therefore $\wtl{X}^{(n)}$ and $\wtl{X}^{(n')}$ are disjoint. By Lemma \ref{lemma:separating_set_in_EZ} there exists a compact set $K\subseteq E$ containing $e_0$ such that $\gamma g_n' K$ separates all but finitely many points from the sequence $\varphi(c^{(n)}_m)e_0$ from all but finitely many elements from the sequence $\varphi(c^{(n')}_m)e_0$. 
    Thus, by Corollary \ref{separation-of-two-limits}, the set $\gamma g_n' K$ separates $z^{(n')}$ from $z^{(n)}$. In particular, the sequences $\gamma_n = \gamma g_n e_0$ and $z^{(n)}$ satisfy the assumptions of Lemma \ref{lemma:separated_convergence_2}. Therefore $\lim_{n\to\8}z^{(n)} = z$, which proves that the set of all branch points is dense in $Z$.
\end{proof}

Next lemma describes some vast family of clopen subsets of $Z$.
It will be used (in the next section) in our proof of Theorem $A$,
to verify condition (a5) from the definition of the dense amalgam
for appropriate family $\mathcal{W}$ of the subsets of $Z$.\medskip

\begin{lm}\label{lemma:H_W-saturated}
    Let $\gamma G_y$ be an edge in the Bass-Serre tree $\wtl{X}$, and let $\wtl{X}_0, \wtl{X}_1$ be the two connected components of ${\wtl{X}\setminus\hbox{\rm int}(\gamma G_y)}$. Moreover, let $\mathcal{W} = \bigsqcup_{v\in V_Y} \set{\Lambda\gamma G_v:\gamma\in \Gamma}$. Then, viewing the vertices of $\wtl{X}$ as subsets of $\Gamma$, we get that the set $H = \Lambda(\bigcup V_{\wtl{X}_0})$ is a $\mathcal{W}$-saturated clopen in $Z$. 
\end{lm}

\begin{proof}
    Firstly let us define an auxiliary set $H' = \Lambda(\bigcup V_{\wtl{X}_1})$. We have 
    $$H\cup H' = \Lambda(\bigcup V_{\wtl{X}_0})\cup \Lambda(\bigcup V_{\wtl{X}_1}) = \Lambda(\bigcup V_{\wtl{X}}) = Z,$$
    and moreover each of the sets $H,H'$ is closed in $Z$ since they are both defined as limit sets. To show that $H$ is also open it suffices to prove that $H, H'$ are disjoint. Let $z \in H$ and $z'\in H'$ be any two points. Then we can express them as limits 
    $$z = \lim_{n\to \8} \gamma_ne_0$$
    and
    $$z' = \lim_{n\to \8} \gamma_n'e_0,$$
    where $\gamma_n\in \bigcup V_{\wtl{X}_0}$ and $\gamma_n'\in \bigcup V_{\wtl{X}_1}$. By Lemma \ref{lemma:separating_set_in_EZ} there exists a compact set $K$ that separates all but finitely many points of form $\delta e_0$ where $\delta\in \bigcup V_{\wtl{X}_0}$ from all but finitely many points of form $\delta' e_0$ where $\delta'\in \bigcup V_{\wtl{X}_1}$. In particular that set $K$ separates all but finitely many points of the sequence $\gamma_ne_0$ from all but finitely many points of the sequence $\gamma_n'e_0$, and therefore, by Corollary \ref{separation-of-two-limits}, $z\neq z'$.

    To show that $H$ is $\mathcal{W}$-saturated, we need to show that any $W = \Lambda \gamma G_v \in \mathcal{W}$ is either contained in $H$ or disjoint with $H$. For that we need to consider two cases. Firstly, suppose that $\gamma G_v \in V_{\wtl{X}_0}$. Then we have
    $$W = \Lambda \gamma G_v \subseteq \Lambda (\bigcup V_{\wtl{X}_0}) = H.$$
    On the other hand, if $\gamma G_v\not\in V_{\wtl{X}_0}$, then $\gamma G_v\in V_{\wtl{X}_1}$, and therefore
    $$W = \Lambda \gamma G_v \subseteq \Lambda (\bigcup V_{\wtl{X}_1}) = H'.$$
    By disjointness of $H$ and $H'$ we have $W \cap H = \emptyset$.
\end{proof}

\subsection{Finite vertex groups}\label{Subsection:boundary_of_virt_free_groups}

In this subsection we study the $E\mathcal{Z}$-boundary of a fundamental group of graph of groups when all of the vertex and edge groups are finite. 
We show that in this case the $E\mathcal{Z}$-boundary
can be identified with the boundary $\partial_\infty\wtl{X}$
of the corresponding Bass-Serre tree.
The corresponding concept of the boundary of a tree
is understood here in the standard sense, recalled inside the proof of the corresponding result, i.e. Lemma \ref{lemma:virtually_free_case}. 
As a consequence of this result, we deduce that any $E\mathcal{Z}$-boundary of a virtually free group is homeomorphic to the Cantor space $\mathcal{C}$ (see Corollary \ref{corollary:all_vertex_groups_finite}).
Even though this last result seems to belong to folklore, we include its proof for completeness. 

\medskip

\begin{lm}\label{lemma:virtually_free_case}
    Let $\Gamma$ be the fundamental group of a graph of groups $(\mathcal{G},Y)$ with all vertex and edge groups finite. Then any $E\mathcal{Z}$-boundary of $\Gamma$ is canonically homeomorphic to the boundary $\partial_\infty\wtl{X}$ of the corresponding Bass-Serre tree $\wtl{X}$.
\end{lm}

\begin{proof}
    Let $\gamma_0 G_{v_0}$ be a fixed basepoint in a Bass-Serre tree and let $(\ol{E},Z)$ be any $E\mathcal{Z}$-structure for $\Gamma$. Since each vertex group $G_v$ is finite, then by Lemma \ref{lemma:vertex_and_branch_points_are_disjoint} there aren't any vertex points in $Z$. Therefore each point in $Z$ is represented as a limit $\lim \varphi(c_n)$, where $c = (c_1, c_2, \ldots)$ is a branch starting from $\gamma_0G_{v_0}$. On the other hand, each point in the boundary $\partial_\8\wtl{X}$ of the Bass-Serre tree $\wtl{X}$ also corresponds to a branch starting at our fixed base point $\gamma_0G_{v_0}$. We will show that the map $\Phi\colon \partial_\8\wtl{X} \to Z$ given by 
    $$c \mapsto \lim_{n\to\8} \varphi(c_n)$$
    is a homeomorphism. Since both spaces $\partial_{\8}\wtl{X}$ and $Z$ are compact it is enough to show that $\Phi$ is a bijective open map.

    Since every point in $Z$ is of the form $\lim_{n\to\8} \varphi(c_n)$ for some branch $c_n$ then we know that $\Phi$ is a surjection. Moreover, by Lemma \ref{lemma:different_branch_points}, the function $\Phi$ is also an injection. Finally to show that $\Phi$ is an open map it is enough to check that any image of a base open set from $\partial_\8\wtl{X}$ is also open in $Z$. The base open sets in $\partial_\8\wtl{X}$ are sets of form
    $$U_{\gamma G_y} = \set{c\in \partial_\8\wtl{X}: \exists n\in \NN\ c_n = \gamma G_y}.$$
    We will show that $\Phi(U_{\gamma G_y}) = \Lambda \bigcup V_{\wtl{X}_0}$, where $\wtl{X}_0$ is a connected component of $\wtl{X}\setminus \hbox{\rm int}(\gamma G_y)$ not containing $\gamma_0 G_{v_0}$. Of course all but finitely many edges of each of the branches $c\in U_{\gamma G_y}$ lie in the subgraph $\wtl{X}_0$, thus for each of those edges $c_n$ we have (by referring to Remark \ref{remark:containment_of_edge})
    $$\varphi(c_n) \in \alpha(c_n)\cup \omega(c_n) \subseteq \bigcup V_{\wtl{X}_0}.$$
    Therefore in particular 
    $$\Phi(c) = \lim_{n\to\8}\varphi(c_n) \in \Lambda \bigcup V_{\wtl{X}_0},$$
    which proves that $\Phi(U_{\gamma G_y})\subseteq \Lambda \bigcup V_{\wtl{X}_0}$. Conversely, let $z\in \Lambda \bigcup V_{\wtl{X}_0}$ be any point. Any point in $Z$ is a branch point, therefore we can define $z = \lim_{n\to\8}\varphi(c_n)$ and suppose on the contrary that $c\not\in U_{\gamma G_y}$. Then the branch $c$ do not contain edge $\gamma G_y$, therefore $c$ lies in $\wtl{X}\setminus \wtl{X}_0$. From this we can conclude that $\varphi(c_n)\in \alpha(c_n)\cup \omega(c_n)\in \wtl{X}\setminus \wtl{X}_0$ By Lemma \ref{lemma:separating_set_in_EZ} for any fixed base point $e_0\in E$ there exists a compact set $K\subseteq E$ such that $\gamma K$ separates all but finitely many elements of form $\delta e_0$ from all but finitely many elements of form $\delta' e_0$, where $\delta \in \bigcup V_{\wtl{X}_0}$ and $\delta' \in \bigcup V_{\wtl{X}\setminus\wtl{X}_0}$. In particular, if $e_n$ is any sequence of elements from the set $(\bigcup V_{\wtl{X}_0})e_0$ that converges to a point of $Z$, then all but finitely many elements of $e_n$ are separated in $E$ from all but finitely many elements $\varphi(c_n)e_0$ by the set $\gamma K$, and thus, by Corollary \ref{separation-of-two-limits}, 
    $$\lim_{n\to\8}\varphi(c_n)e_0 \not\in \Lambda (\bigcup V_{\wtl{X}_0}).$$
    This contradiction proves that $\Phi(U_{\gamma G_y})\supseteq\Lambda \bigcup V_{\wtl{X}_0}$, which consequently means that $\Phi(U_{\gamma G_y})=\Lambda \bigcup V_{\wtl{X}_0}$. By Lemma \ref{lemma:H_W-saturated}, this means that $\Phi(U_{\gamma G_y})$ is open, which proves that $\Phi$ is an open map.
\end{proof}

By referring to Lemma \ref{lemma:behaviour_of_nonelementary_gog}(ii), we can further analyse the case of non-elementary graphs of groups with finite vertex and edge groups. In that case, there is a vertex $v$ in a graph of groups such that any lift of $v$ splits $\wtl{X}$ into at least $3$ infinite components. Moreover, any  such infinite component necessarily contains another lift of a vertex $v$. In particular, in that case our Bass-Serre tree $\wtl{X}$ is easily seen to be locally finite and perfect (where by the latter we mean that any branch in $\wtl{X}$ shares arbitrarily large initial parts with other branches), therefore its boundary is a Cantor set.\medskip 

\begin{cor}\label{corollary:all_vertex_groups_finite}
    Let $(\mathcal{G},Y)$ be a non-elementary graph of groups with finite vertex and edge groups, and let $\Gamma$ be its fundamental group. Then any $E\mathcal{Z}$-boundary of $\Gamma$ is homeomorphic to the Cantor space $\mathcal{C}$.
\end{cor}

\section{Proofs of the main results (Theorems A and B)}\label{Section:main_result}

In this section we present proofs of Theorems A and B of the introduction. 
\medskip


\noindent
{\bf Proof of Theorem A}

    Firstly, note that for the case where all of the vertex groups are finite, by Corollary \ref{corollary:all_vertex_groups_finite} we know that $Z\cong \mathcal{C} \cong \damalgam(\emptyset)$. Therefore it is enough to consider the case where at least one of the vertex groups is infinite; in that case by Remark \ref{Remark:definition_of_dense_amalgam} we have
    $$\damalgam_{v\in V_Y}\Lambda G_v \cong \damalgam_{v\in V_Y^+}\Lambda G_v$$
    where $V_Y^+ \coloneqq \set{v\in V_Y: \Lambda G_v\neq \emptyset} = \set{v\in V_Y: G_v\text{ is infinite}}$. 
    
    For each vertex $v\in V_Y^+$ we define a countable family $\mathcal{W}_v = \set{\Lambda\gamma G_v: \gamma\in \Gamma}$. Now put $\mathcal{W} = \bigsqcup_{v\in V_Y^+}\mathcal{W}_v$. We will show that so described families $\mathcal{W}_v$ and $\mathcal{W}$ satisfy conditions (a1)-(a5) from the definition of the dense amalgam (Definition \ref{Definition:dense_amalgam}).
    \begin{itemize}
        \item[(a1)] First, observe that any $W\in \mathcal{W}_v$ is homeomorphic to $\Lambda G_v$. Indeed, if $W = \Lambda \gamma G_v$ then we have 
        $$\Lambda G_v \cong \gamma \Lambda G_v = \Lambda \gamma G_v,$$
        where the homeomorphism above is a consequence of the fact that $\Gamma$ acts on $Z$ by homeomorphisms, while the equality above is a consequence of the fact that the translation by $\gamma$ is continuous on the entire $\ol{E}$. It remains to show, that any two distinct $W,W'\in \mathcal{W}$ are disjoint. Let $w\in W = \Lambda\gamma G_v$ and $w'\in W' = \Lambda \gamma' G_{v'}$ and let $\gamma'' G_y$ be an edge on the path
        in the Bass-Serre tree $\wtl{X}$
        joining $\gamma G_v$ and $\gamma'G_{v'}$. By Fact \ref{lemma:final_separation}, there exists a compact set $K$ such that $\gamma''K$ separates all but finitely many points from $(\gamma G_v)e_0$ from all but finitely many points from $(\gamma'G_{v'})e_0$. Therefore, by Corollary \ref{separation-of-two-limits}, we have $w\neq w'$, which proves disjointness of $W$ and $W'$.
        
        \item[(a2)] To prove nullness of the family $\mathcal{W}$, it is obviously sufficient
to show nullness of $\mathcal{W}_v$ for each $v\in V_Y^+$.
Fix $v\in V_Y^+$ and suppose that the family
$\mathcal{W}_v=\{ \Lambda\gamma G_v:\gamma\in\Gamma \}$
is not null. Then there is $D>0$ and a sequence of pairwise distinct
cosets $\gamma_nG_v$ such that $\diam(\Lambda\gamma_nG_v)>D$
for all $n$. By passing to a subsequence, we can then find two sequences
of elements $z_n$ and $z_n'$, with $z_n,z_n'\in\Lambda\gamma_nG_v$
for all $n$, which converge to distinct points $z_0,z_0'\in Z$, respectively.
        
Now, viewing the cosets $\gamma_nG_v$ as vertices in the Bass-Serre
tree $\widetilde X$ of $(\mathcal{G},Y)$, by applying K\H{o}nig's lemma and passing to a subsequence,
we can assume that either the vertices $\gamma_nG_v$ lie in a monotonic
order on some branch in $\widetilde X$, or there is a vertex $v_0$ of
$\widetilde X$ such that all $\gamma_nG_v$ are distinct from $v_0$
and the edges adjacent to $v_0$ on the paths $[v_0,\gamma_nG_v]$
in $\widetilde X$ are pairwise distinct. In either case,
for each $n$ there is an edge $\varepsilon_n$ of $\widetilde X$
which separates $\gamma_nG_v$ from all $\gamma_kG_v$ with $k>n$,
and the edges $\varepsilon_n$ are pairwise distinct.
For each $n$, fix $\gamma_n''$ such that $\varepsilon_n=\gamma_n''G_{y_n''}$
(where $y_n''$ is the appropriate edge of $Y$).
By finiteness of $Y$, we can assume by passing to a subsequence that
all $y_n''$ coincide, and hence
the elements $\gamma_n''$ are pairwise distinct. 

{\bf Claim.}
{\it There is a compact $K\subseteq E$ such that for each $n$ the
translate $\gamma_n''K$ separates $z_n$ from all $z_k$ with $k>n$
(and at the same time separates $z_n'$ from all $z_k'$ with $k>n$). }

To prove the claim, let $e_0$ and $K$ be as in Lemma  
\ref{lemma:separating_set_in_EZ}. 
It follows then from  
Fact \ref{lemma:final_separation} that for each $n$ and any $k>n$,
the set $\gamma_n''K$
separates all but finitely many points from the set $(\gamma_nG_v)e_0$
from all but finitely many points from the set $(\gamma_kG_v)e_0$.
Since $z_n$ is the limit of a sequence of points from $(\gamma_nG_v)e_0$,
and similarly $z_k$ is the limit of a sequence of points from 
$(\gamma_kG_v)e_0$, we deduce from Corollary 
\ref{separation-of-two-limits} that $\gamma_n''K$
separates $z_n$ from $z_k$, as required. The same argument shows
also that $\gamma_n''K$
separates $z_n'$ from $z_k'$, hence the claim.

Without loss of generality, we assume now that the sequence 
$\gamma_n''$ converges to a point $z\in Z$.
In view of Lemma 
\ref{lemma:separated_convergence_2}, convergence of $\gamma_n''$ to $z$
together with the claim above imply that $\lim z_n=z$,
and similarly that $\lim z_n'=z$.
The equality $z_0=z_0'$ resulting from these observations
contradicts an earlier 
assumption that $z_0\ne z_0'$.
This concludes the verification of condition (a2).

        \item[(a3)] Let $W\in \mathcal{W}$ be any set from the family $\mathcal{W}$. Since $W$ is defined as the limit set of a coset $\gamma G_v$ for some $\gamma\in \Gamma$ and some $v\in V_Y$, therefore all of the points in $W$ are $G_v$-vertex points.  From Lemma \ref{lemma:vertex_and_branch_points_are_disjoint} we know that none of the points from $W$ is a branch point, and from Lemma \ref{lemma:branch_points_dense} we know that there is a sequence of branch points convergent to any $w\in W$. Thus $W$ is a boundary subset of $Z$.
        
        \item[(a4)] Firstly, for any branch point $z\in Z$, we will find a sequence of points from $\bigcup \mathcal{W}_v$ convergent to $z$. Let $z = \lim_{n\to\8}\varphi(c_n)e_0$ (where $\varphi$ is a function from Definition \ref{Def:phi}) and for any $n$ let $z_n$ be a point from the limit set $\Lambda \varphi(c_n)G_v$. By Fact \ref{fact:distance_edge_and_vertex}, we know that the distance between the vertex $\varphi(c_n)G_v$ and the edge $c_n$ in $\wtl{X}$ is at most $\abs{\abs{O}_Y}$. In particular, the edge $c_{n+\abs{\abs{O}_Y}+1}$ separates the vertex $\varphi(c_n)G_v$ from any vertex $\varphi(c_m)G_v$, where $m\geqslant n+2\abs{\abs{O}_Y}+2$. From Fact \ref{lemma:final_separation} we know that there exist a compact set $K\subseteq E$ such that $\varphi(c_{n+\abs{\abs{O}_Y}+1})K$ separates all but finitely many points $\delta e_0$ where $\delta\in \varphi(c_n)G_v$ from all but finitely many points $\delta' e_0$ where $\delta'\in \varphi(c_m)G_v$. Since both points $z_n$ and $z_m$ can be described as limits of points from $\varphi(c_n)G_ve_0$ and $\varphi(c_m)G_ve_0$ respectively, then by Corollary \ref{separation-of-two-limits} the compact set $\varphi(c_{n+\abs{\abs{O}_Y}+1})K$ also separates $z_n$ from $z_m$ in $\ol{E}$. 
        It follows that the sequences $\varphi(c_n)$ and $z_n$ satisfy the assumptions of Lemma \ref{lemma:separated_convergence_2}, so therefore $\lim_{n\to\8} z_n = z$. 
        
        From the above argument we know that each of the branch points can be obtained as a limit of points from $\bigcup \mathcal{W}_v$, therefore the set of all branch points is contained in the closure of $\bigcup \mathcal{W}_v$. Since the set of all branch points is dense in $Z$
        (see Lemma \ref{lemma:branch_points_dense}), we conclude that
        $$\text{cl}_Z\left(\bigcup \mathcal{W}_v\right) = \text{cl}_Z\left(\text{cl}_Z\left(\bigcup \mathcal{W}_v\right)\right) \supseteq \text{cl}_Z\left(\set{\lim_{n\to\8}\varphi(c_n)e_0: c_n\text{ is a branch in }\wtl{X}}\right) = Z,$$ 
        therefore $\bigcup \mathcal{W}_v$ is dense in $Z$.
        
        \item[(a5)] Let $z, z'\in Z$ be two points that do not belong to the same subset $W\in \mathcal{W}$, and let $K\subseteq E$ be the compact set from the assertion of Lemma \ref{lemma:separating_set_in_EZ}. We need to define a $\mathcal{W}$-saturated clopen set $H$ which separates $z$ from $z'$. By Lemma \ref{lemma:classification_in_Z}, it suffices to consider $3$ cases regarding whether those points are vertex points or branch points, and we will  indicate the appropriate set $H$ separately in each of those cases.

        Suppose first that $z$ is a $G_v$-point and $z'$ is a $G_{v'}$-point. Let $\gamma G_v$ and $\gamma'G_{v'}$ be the cosets of $G_v, G_{v'}$ such that $z\in \Lambda \gamma G_v$ and $z'\in \Lambda \gamma' G_{v'}$. By the assumption that $z,z'$ do not belong to the same set $W\in \mathcal{W}$ we know that $\gamma G_v \neq \gamma'G_{v'}$. Thus, those cosets correspond to different vertices in the Bass-Serre tree $\wtl{X}$, so let $\gamma''G_{y''}$ be an edge of $\wtl{X}$ lying on the unique geodesic path joining $\gamma G_v$ and $\gamma'G_{v'}$. By removing the interior of the edge $\gamma''G_{y''}$ from $\wtl{X}$ we get two connected subgraphs $\wtl{X}_0,\wtl{X}_1$ (we put $\wtl{X}_0$ to be the one of those subgraphs containing $\gamma G_v$). Let now $M$ be the union of all cosets $\delta G_v\in V_{\wtl{X}_0}$ (i.e. cosets representing the vertices of $\wtl{X}_0$). We put $H = \Lambda M$, and we note that $\gamma G_v \subseteq H$, and thus $z\in H$. Moreover, by Lemma \ref{lemma:separating_set_in_EZ} the set $\gamma''K$ separates all but finitely many points of form $\delta' e_0$ where $\delta' \in \gamma'G_{v'}$ from all but finitely many points of form $\delta e_0$ where $\delta\in M$. Therefore, by Corollary \ref{separation-of-two-limits}, we get that $z'\not\in \Lambda M = H$.

        Suppose now that one of the points from the pair $\set{z,z'}$ is a $G_v$-vertex point and the other is a branch point. Without loss of generality, assume that $z' = \lim_{n\to \8}\varphi(c_n)e_0$ for some branch $c = (c_1, c_2, \ldots)$, and that $z\in \Lambda \gamma G_v$. Let $n\in \NN$ be such that $\omega(c_n)$ and $\gamma G_v$ are contained in different connected components of $\wtl{X}\setminus \hbox{int}(c_n)$. Let $\wtl{X}_0, \wtl{X}_1$ be those connected components (again, we assume that $\wtl{X}_0$ is the one component containing the vertex $\gamma G_v$) and let $M=\bigcup V_{\wtl{X}_0}$. Like in the previous case we put $H = \Lambda M$. The same argument as in the previous case proves that $z\in H$; it is left to show that $z'\not\in H$. By definition we know that all of the edges $c_m$ where $m>n$ belong to the subgraph $\wtl{X}_1$. By Remark \ref{remark:containment_of_edge} we know that $\varphi(c_m) \in \alpha(c_m)\cup\omega(c_m)$, where vertices $\alpha(c_m), \omega(c_m)$ are treated as their underlying cosets in $\Gamma$. Of course sequence $\varphi(c_m)$ has infinitely many distinct values. Applying Lemma \ref{lemma:separating_set_in_EZ} to the previous observations we conclude that the set $\varphi(c_n)K$ separates all but finitely many points $\varphi(c_m)e_0$ from all but finitely many points from the set $Me_0$. In particular, by Corollary \ref{separation-of-two-limits}, we conclude that $z'\not\in H$.

        Suppose finally that $b = (b_1, b_2, \ldots)$ and $c = (c_1, c_2, \ldots)$ are some branches in $\wtl{X}$ starting at the same base vertex, and let $z = \lim_{n\to \8} \varphi(b_n)e_0$, $z' = \lim_{n\to \8} \varphi(c_n)e_0$ be the corresponding branch points. Since $z\neq z'$, we have $b\neq c$, so we can choose $n\in \NN$ such that the entire path $b$ is contained in a  different connected component of $\wtl{X}\setminus\hbox{int}(c_n)$ than $\omega(c_n)$. Let $\wtl{X}_0, \wtl{X}_1$ be the connected components of $\wtl{X}\setminus\hbox{int}(c_n)$ (with $\wtl{X}_0$ being the component containing vertex $\alpha(c_n)$), and let $M=\bigcup V_{\wtl{X}_0}$. We put $H = \Lambda M$. Note that for every $m>n$ an edge $b_m$ is contained in $\wtl{X}_0$ and that by Remark \ref{remark:containment_of_edge} $\varphi(b_m)\in \alpha(b_m)\cup \omega(b_m)\subseteq M$. Therefore $z = \lim_{m\to\8}\varphi(b_m)e_0 \in \Lambda M = H$. Moreover, the argument as in the previous case shows that $z'\not\in H$.

        By Lemma \ref{lemma:H_W-saturated}, in all of the above three cases the indicated set $H$ is a $\mathcal{W}$-saturated clopen, and this concludes the verification of condition (a5).
    \end{itemize}

This completes also the whole proof of Theorem A. \qed

We now pass to the proof of our second main result.\medskip




\noindent
{\bf Proof of Theorem B}

By Freudenthal-Hopf theorem, $\Gamma$ is 0-ended, 1-ended, 2-ended or
$\infty$-ended. If $\Gamma$ is 0-ended, it is obviously finite, and then its 
any $E\mathcal{Z}$-boundary $Z$ is empty.
If $\Gamma$ is 1-ended then $Z$ is connected by
Lemma \ref{limit_set_connected}. If $\Gamma$ is 2-ended, it is virtually infinite cyclic
(see e.g. Thorem 3.1 in \cite{Wall}), and hence it is the fundamental group
of a graph of groups of finite order (compare Theorem 3.5 in \cite{Wall}). 
The Bass-Serre tree of this 
graph of groups
is then 2-ended as well  (because $\Gamma$
acts on it geometrically, and because the number of ends is a quasi-isometry invariant of proper geodesic spaces). It follows from Lemma \ref{lemma:virtually_free_case} 
that $Z$ is then a doubleton.
Finally, if $\Gamma$ is $\infty$-ended, we consider its Dunwoody-Stallings
decomposition as the fundamental group of a graph of groups
$\mathcal{G}$
with finite edge groups and with vertex groups that are either finite
or finitely generated and 1-ended (existence of such a decomposition has been justified by the accessibility result of M. Dunwoody \cite{Dunwoody}). 

We claim that this graph of groups is then non-elementary.
To see this, suppose a contrario that some sequence of successive elementary collapses applied to $\mathcal{G}$ converts it into
a new graph of groups $(\mathcal{G}',Y')$ which has one of the three
elementary forms listed in Definition \ref{Definition:non-elementary_gog}.
Observe that the fundamental group of $\mathcal{G}'$ still coincides with $\Gamma$, and that the families of the vertex groups and the edge groups of $\mathcal{G}'$ are contained in the corresponding
families for $\mathcal{G}$, so that the edge groups of $\mathcal{G}'$
are still finite, and each vertex group of $\mathcal{G}'$ is either finite or 1-ended. Now, if the underlying graph $Y'$ of $\mathcal{G}'$ coincides with a single vertex, the corresponding vertex group has to coincide with $\Gamma$, which is infinitely-ended, contradicting its finiteness or 1-endedness. If $Y'$ has a single edge (either loop or non-loop), the conditions from Definition \ref{Definition:non-elementary_gog} easily imply that the Bass-Serre tree $\wtl{X}'$ of $\mathcal{G}'$ coincides then with an infinite line, and all vertex groups of $\mathcal{G}'$ are finite. It means that $\Gamma=\pi_1(\mathcal{G}',Y')$ acts on a line, properly and cocompactly, and this implies that $\Gamma$ has 2 ends,
contradicting its assumed $\infty$-endedness. It follows that $\mathcal{G}$ is indeed non-elemntary.

By Theorem A, $Z$ is then homeomorphic to
the dense amalgam of the limit sets in $Z$ of the vertex groups
of $\mathcal{G}$ (viewed as subgroups of $\Gamma$). If all the vertex 
groups are finite, Theorem A assures that $Z$ is a Cantor set,
and it is thus homeomorphic to the dense amalgam of a single space
being a singleton. In the remaining case, $Z$ is homeomorphic
to the dense amalgam of the limit sets of the infinite vertex groups 
of $\mathcal{G}$, which are all 1-ended. Since by 
Lemma \ref{limit_set_connected} 
all these limit sets are connected, we get that in any case $Z$
is the dense amalgam of a family of connected spaces. 

These arguments together yield parts 1-3 of Theorem B,
which completes the proof.  \qed

\medskip
\begin{rk}
Note that Theorem B still holds true, with the same proof, if we assume that
instead of being finitely presented, the group $\Gamma$ is finitely generated and accessible (where the latter means it can be raelized as the fundamental group
of a graph of groups whose all edge groups are finite, and the vertex groups are
either finite or 1-ended).
\end{rk}



\section{Final remarks and questions}\label{Section:final_conclusions}

The framework of $E\mathcal{Z}$-boundaries is very general and includes various commonly known frameworks, 
such as Gromov boundaries, CAT(0) boundaries, and systolic boundaries. 
Therefore our main results instantly translate into each of the just mentioned frameworks.

In order to fully understand the boundaries of infinitely ended groups in terms of their 1-ended factors
(for terminal splittings over finite subgroups), one needs to understand the limit sets of these factors
in the corresponding frameworks. In case of Gromov hyperbolic groups, all the factors have to be quasi-convex,
and Gromov
hyperbolic themeselves, so their limit sets (in the Gromov boundary of the whole group)
have to be homeomorphic with their own Gromov boundaries (compare Theorem 0.3 (2) in \cite{dense-amalgam}).
One could ask whether that kind of observation is also feasible in other discussed frameworks. 
We state the corresponding questions as the the question below.\medskip

\begin{oprob}\label{Open_problem_1}
    Suppose that $\Gamma = \pi_1(\mathcal{G},Y)$ is a fundamental group of a graph of group with finite edge groups, and that it admits an $E\mathcal{Z}$-structure $(\ol{E},Z)$ (respectively: CAT(0) structure, systolic structure). 
Do the vertex groups also admit $E\mathcal{Z}$-structures (CAT(0) structure, systolic structure)? 
If so, can their $E\mathcal{Z}$-structures (CAT(0) structures, systolic structures) be chosen in such a way 
that the corresponding boundary of each such vertex group $G_v$ is homeomorphic to its limit set $\Lambda G_v\subseteq Z$?
\end{oprob}

In view of Theorem 0.3 in \cite{dense-amalgam},
positive answers to both of the questions above (for a given type of boundary) would give a complete description of boundaries
of an infinitely ended group in terms of the boundaries
of its 1-ended factors.
On the other hand, negative answers would lead to further questions concerning possible forms of limits sets of groups appearing as vertex groups in the situations as above.

As a next comment, we note that one could try to generalise our Theorem A to the case where the edge groups in the discussed graphs of groups are allowed to be infinite.
In such a case, the boundary of the whole group $\pi_1\mathcal{G}$
still contains copies of the limit sets $\Lambda G_v$ of the vertex subgroups, but this time these copies overlap along
the subsets corresponding to the limit sets of the edge groups.
The whole picture thus becomes much more complicated.
Some example provided by M. Bestvina (Example 3.1 in \cite{bestvina}) shows that a graph of groups with quite
well behaved vertex groups and edge groups (where the latter are infinite) may have several significantly distinct and seemingly unrelated forms
of $E\mathcal{Z}$-boundaries.
We suspect that any reasonable generalisation of Theorem A
requires imposing some further restrictions on the structure of the considered graphs of groups.\medskip

\begin{oprob}
    How can one generalize Theorem A to any appropriate cases where the edge groups of the considered graphs of groups are not necessarily finite?
\end{oprob}

Passing to our final remark, observe that Theorem B and Corollary \ref{corollary:all_vertex_groups_finite} show some examples of groups with unique $E\mathcal{Z}$-boundaries (finite groups, virtually infinite cyclic groups and virtually free groups). Some other examples 
include the Poincare duality groups of dimensions $n\leqslant3$,
whose $E\mathcal{Z}$-boundaries are homeomorphic to $S^{n-1}$ (see Theorem 2.8 and Remark 2.9 in \cite{bestvina}). On the other hand, the already mentioned example of Bestvina (Example 3.1, \cite{bestvina}) shows that this cannot be the case in general even for some well-behaved classes of groups (such as Gromov hyperbolic groups). The uniqueness of boundaries is a  widely studied topic for both CAT(0) and some other subtypes of $E\mathcal{Z}$-boundaries such as boundaries of Helly groups (see e.g. \cite{crooke-kleiner, Ruane, Schreve-Stark, Danielski}). This motivates us to state the following.\medskip

\begin{oprob}
    Which groups have unique (up to homeomorphism) $E\mathcal{Z}$-boundaries?
\end{oprob}

Note that, if the answers to both questions in \ref{Open_problem_1} were positive, then an infinitely ended group 
would have a unique $E\mathcal{Z}$-boundary if and only if all of its one ended factors (of Dunwoody-Stallings decomposition) had a unique $E\mathcal{Z}$-boundary. 

\bibliographystyle{amsalpha}
\bibliography{mybibliography}

\vfill

\indent (M. Kandybo) \textsc{Department\hspace{0.2cm} of Mathematical Sciences, University of Copenhagen, Universitetsparken 5, 2100 Copenhagen, Denmark}\\
\textit{Email address:} \href{mailto:mk@math.ku.dk}{mk@math.ku.dk}

(J. Świątkowski) \textsc{Mathematical Institute, University of Wrocław, pl. Grunwaldzki 2, \mbox{50-384} Wrocław}\\
\textit{Email address:} \href{mailto:jacek.swiatkowski@math.uni.wroc.pl}{jacek.swiatkowski@math.uni.wroc.pl}

\end{document}